\documentclass[12pt]{article}        

 \usepackage[T1]{fontenc}              
 \usepackage[width=16cm, height=22cm]{geometry}   

 \usepackage{mathtools} 
 \usepackage{amssymb}          
\usepackage{hyperref} 
 \usepackage[amsmath,thmmarks,hyperref]{ntheorem} 

 \usepackage{icomma}                   
 \usepackage{enumerate}          
 \usepackage{color}                 
 \usepackage{setspace}                 
 \usepackage{needspace}               
 \usepackage{url}                    
 \usepackage[mathscr]{euscript}              
 \usepackage{esint}    
 \usepackage[numbers,square]{natbib}

\usepackage{tikz}
\usetikzlibrary{arrows.meta}
\usetikzlibrary{cd}
 \numberwithin{equation}{section}
 
\theoremstyle{nonumberplain}  
\theoremheaderfont{\itshape}  
\theorembodyfont{\normalfont}  
\theoremseparator{.}  
\theoremsymbol{$\Box$}
\newtheorem{proof}{Proof} 
  
\theoremstyle{plain}  
\theoremheaderfont{\normalfont\bfseries}  
\theorembodyfont{\itshape}  
\theoremsymbol{~}
\theoremseparator{.}  
\newtheorem{proposition}{Proposition}[section]  
  
\newtheorem{corollary}[proposition]{Corollary}  
\newtheorem{lemma}[proposition]{Lemma}  
\newtheorem{theorem}[proposition]{Theorem}   

\theorembodyfont{\normalfont}  
 
\newtheorem{remark}[proposition]{Remark}
\newtheorem{example}[proposition]{Example}

\newtheorem{definition}[proposition]{Definition} 
\newtheorem{notation}[proposition]{Notation} 

\theoremstyle{nonumberplain}
\theoremheaderfont{\normalfont\bfseries}  
\theorembodyfont{\itshape}  
\theoremsymbol{~}
\theoremseparator{.}  
\newtheorem{theoremA}[proposition]{Theorem A}  
\newtheorem{theoremB}[proposition]{Theorem B} 
\newtheorem{theoremC}[proposition]{Theorem C} 
\newtheorem{theoremD}[proposition]{Theorem D} 
\newtheorem{theoremE}[proposition]{Theorem E} 

\newcommand{\R}{\mathbb{R}}

\newcommand{\N}{\mathbb{N}}
\newcommand{\Z}{\mathbb{Z}}
\newcommand{\A}{\mathcal{A}}

\newcommand{\C}{\mathbb{C}}
\newcommand{\dd}{\mathrm{d}}

\newcommand{\ind}{\operatorname{ind}}
\newcommand{\tr}{\mathrm{tr}}

\newcommand{\Hom}{\mathrm{Hom}}

\newcommand{\ch}{\mathrm{ch}}
\newcommand{\M}{\Mod}

\newcommand{\cc}{\mathbf{c}}

\DeclareMathOperator{\Tr}{\mathrm{Tr}}
\DeclareMathOperator{\im}{\mathrm{im}}

\newcommand{\id}{1}

\newcommand{\dom}{\mathrm{dom}}
\renewcommand{\tilde}{\widetilde}
\renewcommand{\hat}{\widehat}

\newcommand{\CC}{\mathsf{C}}

\newcommand{\DD}{\mathsf{D}}
\newcommand{\NN}{\mathsf{N}}

\newcommand{\T}{\mathbb{T}}

\renewcommand{\L}{\text{\normalfont\sffamily L}}
\newcommand{\B}{\mathsf{B}}

\DeclareMathOperator{\Str}{\mathrm{Str}}

\newcommand{\Lin}{\mathscr{L}}

\newcommand{\Hil}{\mathcal{H}}
\newcommand{\Ch}{\mathrm{Ch}}

\newcommand{\Mod}{\mathscr{M}}

\title{The Chern Character of $\vartheta$-summable Fredholm Modules over dg Algebras and Localization on Loop Space}
\author{Batu G\"uneysu\footnote{Humboldt-Universität zu Berlin. E-mail: gueneysu@math.hu-berlin.de} ~and Matthias Ludewig\footnote{The University of Adelaide. E-mail: matthias.ludewig@adelaide.edu.au}}

\begin{document}

\maketitle

\begin{abstract} 
We introduce the notion of a $\vartheta$-summable Fredholm module over a locally convex dg algebra $\Omega$ and construct its Chern character as a cocycle on the entire cyclic complex of $\Omega$, extending the construction of Jaffe, Lesniewski and Osterwalder to a differential graded setting. Using this Chern character, we prove an index theorem involving an abstract version of a Bismut-Chern character constructed by Getzler, Jones and Petrack in the context of loop spaces. Our theory leads to a rigorous construction of the path integral for $\mathcal{N}=1/2$ supersymmetry which satisfies a Duistermaat-Heckman type localization formula on loop space.
\end{abstract}

\tableofcontents 

\section{Introduction}

The close relation between supersymmetry and index theory was discovered by Alvarez-Gaum\'e \cite{AlvarezGaume, AlvarezGaume2}, motivated by considerations of Witten \cite{Witten2, Witten}, and has gotten a lot of attention ever since (see e.g., \cite{AnderssonDriver, AtiyahCircular, Bismut1, BismutLoc, FineSawin1, FineSawin2, FriedanWindey, GJP, GetzlerThomClass, GetzlerPreprint, JLO, Lott, StolzTeichnerSusy}). In particular, Alvarez-Gaum\'e observed that the path integral for the $\mathcal{N}=1/2$ supersymmetric $\sigma$-model enables a concise proof of the Atiyah-Singer index theorem for the Dirac operator on a spin manifold, by {\em bona fide} generalizing a the Duistermaat-Heckman localization formula to the (infinite-dimensional) manifold of smooth loops. These considerations have been pushed further by Bismut \cite{BismutLoc} to obtain a formal proof of the Atiyah-Singer index theorem for \emph{twisted} Dirac operators (and using $K$-theoretical methods this twisted index theorem actually suffices to derive the most general variant of the index theorem).

Much more than obtaining another proof of the Atiyah-Singer index theorem, these considerations suggest to study the following fundamental question: 
\begin{align*}
&\text{\em To what extend does one have a localization formula}\\
&\text{\em on the loop space of a spin manifold?}
\end{align*}

Note that most of the above mentioned calculations had to remain formal so far, due to the mathematically ill-defined nature of the supersymmetric path integral: Even though the integral in question has a Gaussian nature (cf.\ formula \eqref{IntroFormal} below), the usual approach (which uses continuous loops instead of smooth loops and the Wiener measure for the bosonic integration) makes it notoriously hard to implement supersymmetry.

To explain the above formal results in more detail, let $X$ be a compact, even-dimensional spin manifold, keeping in mind that the spin property of $X$ guarantees the orientability of the loop space $\L X$ \cite{StolzTeichner, Waldorf}. In the differential geometric formulation of Alvarez-Gaum\'e's observations by Atiyah \cite{AtiyahCircular} and Bismut \cite{Bismut1}, the (formally ill-defined) path integral of interest is the integration functional 
\begin{equation}\label{IntroFormal}
I[\xi] \stackrel{\text{formally}}{=} \int_{\L X} e^{-S -\omega} \wedge \xi
\end{equation}
on smooth differential forms $\xi$ on $\L X$, where $S$ is the energy functional and $\omega$ is the canonical presymplectic form on $\L X$. The {\em supersymmetry} of the path integral is then the formula $I[(d - \iota_K)\xi]=0$, where $\iota_K$ denotes the contraction by the velocity vector field $K(\gamma) = \dot{\gamma}$, which generates the natural $S^1$-action on $\L X$ given by rotating each loop.

While the above expression for $I$ is not mathematically well-defined, one nevertheless expects to have the {\em localization formula}\footnote{This in fact implies the supersymmetry property $I[(d-\iota_K)\xi] = 0$.}
\begin{equation} \label{LocalizationFormulaIntro}
  I[\xi] = (2\pi)^{-\dim(X)/2} \int_X \hat{A}(X) \wedge \xi|_{X}
\end{equation}
for differential forms $\xi$ that are equivariantly closed, $(d - \iota_K) \xi = 0$, in analogy with the results of Duistermaat-Heckman and Berline-Vergne \cite{DuistermaatHeckmann, BerlineVergne} on finite-dimensional integrals of this kind (cf. the arguments of Atiyah in \cite{AtiyahCircular}). 
At this point, the proof of the Atiyah-Singer index theorem for twisted Dirac operators becomes a simple consequence of supersymmetry and localization, using the special integrand introduced by Bismut \cite{Bismut1}, the so called \emph{Bismut-Chern character} of the twisting bundle.

\medskip

When attempting to give a rigorous definition of the above functional $I$, an essential observation is that an important class of differential forms on $\L X$ 
%(which includes the Bismut-Chern character) 
is given by \emph{Chen's iterated integrals} \cite{Chen1, Chen2, GJP} and its extensions (c.f.\ \S\ref{SectionPathIntegral}). 
In their simplest form, these iterated integrals arise as the image of a map 
\begin{equation}\label{chmap}
\rho: \CC\bigl(\Omega(X)\bigr) \longrightarrow \Omega(\L X)
\end{equation}
from the \emph{cyclic chain complex} (c.f.~\S\ref{SectionPreliminaries}) over the differential graded algebra $\Omega(X)$ of differential forms on $X$ to $\Omega(\L X)$, the differential forms on $\L X$. 
By the work of Getzler, Jones and Petrack \cite{GJP} and the first-named author \cite{CacciatoriGueneysu}, this space contains in particular the even and odd Bismut-Chern character forms as interesting integrands.
The space of iterated integrals is therefore a natural domain for the desired integral map.

On the other hand, it has been shown by Jaffe, Lesniewski and Osterwalder \cite{JLO} that the Dirac operator on $X$ canonically induces a linear form on the cyclic complex $\CC\bigl(C^{\infty}(X)\bigr)$ over the algebra $C^{\infty}(X)$, the \emph{JLO-cocyle}. In fact, the construction of the JLO-cocycle actually works on an arbitrary ungraded $\vartheta$-summable Fredholm module over a locally convex algebra.
This suggests that some appropriately constructed extension of the JLO-cocycle from $\CC(C^{\infty}(X))$ to $\CC(\Omega(X))$ might serve as a natural candidate for definition of the supersymmetric path integral, once one can make sure that such a linear functional on $\CC(\Omega(X))$ can be pushed forward via $\rho$. Moreover, like in the JLO case, the construction should generalize to an abstract functional analytic framework.

\medskip

In view of the above observations, the aim of article is twofold: firstly, we construct precisely such a differential graded variant of the JLO-cocyle in the abstract setting of \emph{$\vartheta$-summable Fredholm modules over locally convex differential graded algebras}, a concept that is introduced in this paper. We prove that this cocycle fits into a natural noncommutative index theorem and that it enjoys a natural  homotopy invariance. Secondly, we apply this abstract construction to the particular $\vartheta$-summable Fredholm module over (a certain algebraic extension of) $\Omega(X)$ which is canonically induced by the Dirac operator on a spin manifold. This provides a framework that allows to construct the path integral rigorously as a linear functional on Chen interated integrals, and as we will explain in a moment, this functional enjoys the expected localization formula.

\medskip

To explain our results in more detail, and to put it into the context of non-commutative geometry, we recall that a $\vartheta$-summable Fredholm module (as introduced by Connes in \cite{ConnesEntire, ConnesThetaSummable}) consists of a representation $\cc$ of an \emph{ungraded} locally convex algebra $\Omega$ on a $\Z_2$-graded Hilbert space $\Hil$ together with an odd self-adjoint unbounded operator $Q$ on $\Hil$ (subject to certain analytic conditions). Aiming to define the notion of a $\vartheta$-summable Fredholm module over locally convex \emph{differential graded} algebra $\Omega$, it turns out that requiring $\cc$ to be a representation of $\Omega$ is too restrictive. Instead, we just assume that $\cc$ is a degree-preserving linear map from $\Omega$ to $\Lin(\Hil)$ such that
%\footnote{In fact, not even these algebraic properties are needed; dropping these assumptions leads to what we call a {\em weak} Fredholm module, c.f.\ \S\ref{SectionFredholmModules}.}
\begin{equation}  \label{MultiplicativityIntroduction}
  [Q, \cc(f)] = \cc(df), ~~~~\text{and}~~~~ \cc(f\theta) = \cc(f)\cc(\theta), ~~~\cc(\theta f) = \cc(\theta)\cc(f).
\end{equation}
for all $f \in \Omega^0$ and all $\theta \in \Omega$; in particular, $\cc$ is required to be a representation only on the subalgebra $\Omega^0$. These assumptions are satisfied by the $\vartheta$-summable Fredholm module $\Mod^X$ over $X$, where $\cc$ stems from Clifford multiplication and $Q = \mathsf{D}$ is the Dirac operator (c.f.~Example~\ref{ExampleSpinors}) below).

Given an abstract $\vartheta$-summable Fredholm module $\Mod$, the desired extension of the JLO-cocycle is then a linear functional $\Ch_{\Mod}$ on the {\em cyclic chain complex} $\CC(\Omega)$ of Connes, which we call the {\em Chern character} of $\Mod$. The complex $\CC(\Omega)$ carries three differentials: the differential $\underline{d}$, which is induced from the differential $d$ of $\Omega$, the {\em Hochschild differential} $\underline{b}$ and the {\em Connes operator} $\underline{B}$. The Chern character is closed with respect to the {\em total differential} $\underline{d} + \underline{b} + \underline{B}$ (c.f.\ Thm.~\ref{ThmChClosed} below) which acts on linear functionals on $\CC(\Omega)$ by duality, c.f.~\eqref{dualy}.

\begin{theoremA}\label{ThmA}
  The Chern character $\Ch_{\Mod}$ is even and closed, that is,  
  \begin{equation} \label{ClosednessChernIntro}
(  \underline{d} + \underline{b} + \underline{B})\Ch_\Mod= 0.
\end{equation}
\end{theoremA}

Let us briefly outline our construction of $\Ch_\Mod$, in a presentation similar to that of Quillen \cite[\S9]{Quillen} in the ungraded case (c.f.\ also \cite{Perrot}): following Quillen, an odd bar cochain $\omega$ on $\Omega$ (c.f.\ \eqref{FormulaBOmega} below) with values in an algebra $L$ can be seen as a {\em connection form} on the space of $L$-valued bar cochains. Its {\em curvature} is then given by
\begin{equation*} 
  F = \delta \omega + \omega^2,
\end{equation*}
where $\delta$ is the codifferential on $\B(\Omega)$. The observation is now that a $\vartheta$-summable Fredholm module $\Mod$ over $\Omega$ determines such a cochain $\omega = \omega_{\Mod}$ taking values in the linear operators on $\Hil$, with curvature $F_{\Mod}$ (c.f.\ \eqref{ComponentsOfOmega} and \eqref{ComponentsOfF} below); the Chern character is then ultimately defined by
\begin{equation*}
   \Ch_\Mod (\theta_0, \dots, \theta_N) := \Str \bigl(\cc(\theta_0) \Phi^{{\Mod}}(\theta_1, \dots, \theta_N)\bigr), \quad \text{with} \quad \Phi^{{\Mod}} = \exp(-F_{\Mod}),
\end{equation*}
which is much in the spirit of Chern-Weil theory and justifies its name. Closedness of $\Ch_\Mod$, Thm.~A, is then essentially a consequence of the Bianchi identity for $F_{\Mod}$.
However, the fact that $\omega_{\Mod}$ and $F_{\Mod}$ take values in the {\em unbounded} operators, not in $\Lin(\Hil)$, poses severe analytical problems; solving these belongs to the main achievements of this paper.

Related to these analytic issues is the important feature of our Chern character that its homogeneous components decay rapidly enough to extend it to a continuous linear functional on the {\em entire complex} $\CC^\epsilon(\Omega)$. This complex contains infinite sums of homogeneous chains, subject to certain growth conditions modelled on the growth of entire functions (c.f.\ Thm.~\ref{ThmEntire} below). We remark that this growth condition was first introduced on {\em cochains} by Connes in \cite{ConnesEntire}; in our case, it appears naturally on {\em chains}, as in \cite{GetzlerSzenes}. Motivated by the terminology of Meyer \cite{MeyerBook}, we call the dual growth condition satisfied by $\Ch_\Mod$ {\em analytic}, as it is related to the growth of analytic functions (cf. \S\ref{AnaEnt}).

\begin{theoremB} \label{ThmB}
The Chern character $\Ch_\Mod$ is {\em analytic}, that is, it extends to a continuous linear functional on the entire complex $\CC^{\epsilon}(\Omega)$.
\end{theoremB}

To prove an index theorem for our Chern character, we use the Bismut-Chern character $\Ch(p)$ constructed by Getzler, Jones and Petrack in \cite{GJP}. It is an entire chain, which allows to pair it with our Chern character in view of Thm.~B. An obstacle is that $\Ch(p)$ is not a chain over $\Omega$ but over its {\em acyclic extension} 
\begin{equation*}
\Omega_\T := \Omega[\sigma], 
\end{equation*}
where $\sigma$ is a formal variable of degree $-1$ satisfying $\sigma^2 = 0$. Now it turns out that any $\vartheta$-summable Fredholm module $\Mod$ over $\Omega$ can be extended to a $\vartheta$-summable weak Fredholm module $\Mod_\T$ over $\Omega_\T$, giving rise to a Chern character $\Ch_{\Mod_\T}$ (c.f.\ Example~\ref{ExampleExtension} below). The following result (c.f.\ Thm.~\ref{ThmIndex} below) then generalizes the index theorem of Getzler and Szenes \cite[Thm.~D]{GetzlerSzenes}.

\begin{theoremC} \label{ThmC}
For every idempotent $p \in \mathrm{Mat}_n(\Omega^0)$ we have the index formula
\begin{equation} \label{IndexFormulaIntro2}
  \mathrm{ind}(Q_p) = \Ch_{\Mod_\T}\bigl(\Ch(p) \bigr).
\end{equation}
\end{theoremC}

While $\Ch(p)$ is not closed in the complex $\CC^\epsilon(\Omega_\T)$, it is closed in the quotient complex 
\begin{equation*}
\NN^{\T, \epsilon}(\Omega) = \CC^\epsilon(\Omega_\T) / \overline{\DD^\T(\Omega)}
\end{equation*}
of {\em Chen normalized chains}, which is central to our theory (c.f. Def.~\ref{DefExtendedEntireNormalized} below). It turns out that the properties \eqref{MultiplicativityIntroduction} ensure that $\Ch_{\Mod_\T}$ is {\em Chen normalized}, meaning that it vanishes on $\DD^\T(\Omega)$ and therefore descends to a functional on $\NN^{\T, \epsilon}(\Omega)$ (Thm.~\ref{ThmNormalized} below). This leads to a cohomological version of the index theorem involving the $K$-theory of the algebra $\Omega^0$ (Corollary~\ref{ThmHomologicalIndex}).

\medskip

We emphasize that the Chern-Weil type method outlined above constructs the Chern character as a {\em cochain}, not only as a cohomology class. A remarkable feature of this cochain is that the pairing on the right hand side of \eqref{IndexFormulaIntro2} is {\em exactly equal} to the supertrace of the heat operator $e^{-Q_p^2}$ (c.f.\ Prop.~\ref{ThmIndex}); using this, the proof of Thm.~C does not need the homotopy invariance of $\Ch_{\Mod_\T}$, and follows straightforwardly from the McKean-Singer formula. Nevertheless, we prove:

\begin{theoremD}\label{ThmD}
The Chern character $\Ch_{\Mod_\T}$ is invariant under homotopies of $\Mod$. 
\end{theoremD}

In fact, the Chern character is even homotopy-invariant as a {\em Chen normalized} cochain, i.e., as one on the complex $\NN^{\T, \epsilon}(\Omega)$; for a more precise statement, see \S\ref{SectionHomotopy}, where this result is obtained from a Chern-Simons type transgression formula.

\medskip

Let $X$ be a compact, even-dimensional spin manifold $X$ and let $\Mod=\Mod^X$, the Fredholm module over $\Omega(X)$ determined by the Dirac operator (c.f.\ Example~\ref{ExampleSpinors} below). 
Let $\Mod_{\T}^X$ be its acyclic extension and let $\Ch_{\Mod_{\T}^X}$ be the corresponding Chern character. 
In this special case, we have the following \emph{localization principle}, which is proved using Thm.~D in combination with a version of Getzler-rescaling.

\begin{theoremE}
As an element of the complex $\NN^{\T, \epsilon}(\Omega(X))$, the Chern character $\Ch_{\Mod_\T^X}$ is cohomologous to the cochain $\mu_0$ given by the formula
\begin{equation*}
\mu_0(\theta_0, \dots, \theta_N) = \frac{1}{(2 \pi i)^{\dim(X)/2} N!} \int_X \hat{A}(X) \wedge \theta_0^\prime \wedge \theta_1^{\prime\prime} \wedge \cdots \wedge \theta_N^{\prime\prime},
\end{equation*} 
for all $\theta_j = \theta_j^\prime + \sigma \theta_j^{\prime\prime} \in \Omega(X)_\T$, where $\hat{A}(X)$ is the $\hat{A}$-form for $X$; see \eqref{AHatForm}.
\end{theoremE}

To make contact with the supersymmetric path integral formula \eqref{IntroFormal} that served as the initial motivation for our considerations, we note that $\CC(\Omega_\T(X))$ is the natural domain for the extended iterated integral map of Getzler, Jones and Petrack. Pushing forward the Chern character, we obtain a linear functional on a space of differential forms on the loop space. 
Explicitly, $I$ is determined by the formula
\begin{equation} \label{DefIIntro}
I[\rho(c)]= i^{\dim(X)/2} \Ch_{\Mod^X_\T}(c )
\end{equation}
for all $c\in\CC(\Omega_\T(X))$.
Given the basic properties of $\rho$, the well-definedness, respectively supersymmetry of $I$ turns out to follow from the fact that $\Ch_{\Mod^X_\T}$ is Chen normalized, respectively closed. 
Moreover, the localization formula is a consequence of Thm.~E.
We obtain a linear ``integration'' functional $I$ on the space $\Omega_{\mathrm{int}}(\L X) \subset \Omega(\L X)$ of loop space differential forms that can be represented by (entire, extended) iterated integrals. 
This is our candidate for the supersymmetric path integral.

However, the relation of this functional $I$ with the formal expression \eqref{IntroFormal} may be rather unclear at this point.
This gap is filled in the articles \cite{HanischLudewig1, HanischLudewig2} of Hanisch and the second-named author, the results of which we now briefly explain.
In a first step, a stochastic formula for the Chern character is proven: 
Explicitly, using a general version of the Feynman-Kac formula, it is shown in \cite{HanischLudewig1} (see also \cite{boldt2020feynmankac}) that $\Ch_{\Mod_{\T}^X}(\theta_1, \dots, \theta_N)$ can be obtained as the expectation value of the iterated Stratonovich integral 
\begin{equation*}
  \int_0^1 \int_0^{\tau_N} \cdots \int_0^{\tau_2} \mathrm{str}\left(  [\boldsymbol{x}\|_{\tau_1}^0]^\Sigma \prod_{j=1}^N \Bigl(\cc\bigl(\iota_\bullet \theta_j^\prime(\boldsymbol{x}_{\tau_j})\bigr) * \mathrm{d}\boldsymbol{x}_{\tau_j} - \cc(\theta^{\prime\prime}_j(\boldsymbol{x}_{\tau_j})\bigr) \dd \tau_j\Bigr) [\boldsymbol{x}\|_{\tau_{j+1}}^{\tau_j}]^{\Sigma}  \right).
\end{equation*}
Here $\mathbf{x}_\bullet$ denotes a Brownian bridge on $X$ (ending and starting at the same point) and $[\boldsymbol{x}\|_{\tau_{j+1}}^{\tau_j}]^\Sigma$ is the corresponding stochastic parallel transport in the spinor bundle.

On the other hand, in the paper \cite{HanischLudewig2}, a certain functional on differential forms on loop space is constructed, which formally extracts the fiberwise top degree component of the differential form $e^{-\omega} \wedge \xi$. 
This is done by formally generalizing formulas from the finite-dimensional case, which results in an expression in terms of certain Pfaffians.
It is then a rather non-trivial result, proved in the same paper, that this formal top degree functional, as a function on $\L X$, is precisely the deterministic analog of the iterated Stratonovich integral above.

These two results combined therefore suggest to say that {\itshape the integration functional $I$ is formally given by the path integral} \eqref{IntroFormal}, and therefore is the desired rigorous definition of this formula. 
See \cite{ludewig2019construction} for a survey on these connections.

\medskip

We emphasize that the main point of our constructions is \emph{not} to find another proof for the (twisted) Atiyah-Singer index theorem. 
Indeed, a proof along these lines will use Thm.~E, which in turn relies on a version of Getzler's rescaling trick. 
Hence our methods do not give an independent proof.  

Instead, from our point of view, the remarkable fact is that the stochastic formula for the path integral $I$ obtained in \cite{HanischLudewig1, HanischLudewig2} by a suitable interpretation of \eqref{IntroFormal} admits an algebraic-combinatorial description as a non-commutative Chern character. 
Moreover, as seen in \S9, this description allows to verify that $I$ satisfies both supersymmetry and a Duistermaat-Heckmann type localization formula (see Thm.~\ref{ThmLocalization2} below), \emph{as though it was an integral over a finite-dimensional manifold}.
To our knowledge, this is the first example of an infinite-dimensional localization principle in a geometric context.

\medskip

To the best of our knowledge, the programm to use the cyclic cohomology of $\Omega$ and the iterated integral map to construct the supersymmetric path integral has been initiated by Getzler, and this paper (together with \cite{HanischLudewig1, HanischLudewig2}) can be seen as a completion of this programm. Needless to say, our paper is inspired by his work (in particular \cite{GetzlerThomClass, GetzlerPreprint, GetzlerOdd, GetzlerSzenes, GJP}). Moreover, Bismut's results suggest a possible connection between our machinery and the hypoelliptic Laplacian, respectively the hypoelliptic Dirac operator \cite{BismutHypoCot, BismutHypoDirac, BismutHypoOrb}.

\medskip

We would also like to remark that our results could also be stated in the abstract framework of bornological algebras instead of the locally convex setup, as in the work of Meyer \cite{MeyerBook}. However, we believe that the use of locally convex algebras somewhat simplifies our presentation; in particular the entire growth condition on chains takes a somewhat technical form in the bornological setup (cf. \S\ref{AnaEnt}). Moreover, we observe that our main example $\Omega = \Omega(X)$, $X$ a compact manifold, is a nuclear Fr\'echet algebra, so that the von-Neumann and the precompact bornology coincide.

\medskip

The outline of this paper is as follows. First we give the definition of a $\vartheta$-summable Fredholm module over $\Omega$ and give several motivating examples. In \S\ref{SectionPreliminaries}, we introduce the algebraic preliminaries for the constructions in \S\ref{SectionQuantizationMap}, \ref{SectionChernCharacter} and \ref{SectionHomotopy}, where we prove Thm.~A, Thm.~B and Thm.~D. Then in \S\ref{SectionBChOfP}, we discuss the Bismut-Chern character of an idempotent and subsequently prove Thm.~C, in \S\ref{SectionIndexTheorem}. Finally, in \S\ref{SectionPathIntegral}, we discuss the construction of the supersymmetric path integral. In particular, we give a proof of the localization formula \eqref{LocalizationFormulaIntro}, c.f.\ Thm.~\ref{ThmLocalizationFormula} and Thm.~\ref{ThmLocalization2}.

\medskip

\textbf{Acknowledgements.} It is our pleasure to thank J.-M.\ Bismut, S.\ Cacciatori, F.\ Hanisch, N.\ Higson, M.\ Lesch, V.\ Mathai and S.\ Shen for helpful discussions. We are further indebted to the Max-Planck-Institute in Bonn and The University of Adelaide, where parts of the work on this project was conducted. The second-named author was supported by the Max-Planck-Foundation and the ARC Discovery Project grant FL170100020 under Chief Investigator and Australian Laureate Fellow Mathai Varghese.

\medskip

\textit{Conventions.} In this paper, all vector spaces and algebras are over $\C$, and we always work in the category of {\em graded} vector spaces. Every $\Z$-graded vector space is considered as $\Z_2$-graded in terms of the induced even/odd grading, and tensor products and direct sums of graded vector spaces are equipped with their canonically given $\Z$-grading (and thus the induced $\Z_2$-grading). An ungraded vector space will be considered to be $\Z$-graded (and thus $\Z_2$ graded) by declaring all its elements to have degree zero.

If $V, W$ are $\Z_2$-graded vector spaces, the vector space $\mathrm{Hom}(V, W)$ is $\Z_2$-graded in the usual way. Given coefficients $W$, any $A\in \mathrm{End}(V)$ has a $\Z_2$-graded dual map $A^{\vee}\in \mathrm{Hom}(V, W)$ given by
\begin{align}\label{dualy}
(A^{\vee}\ell) (v):= (-1)^{|A||\ell|}\ell (A(v)), \quad \ell\in \mathrm{Hom}(V,W), \quad v\in V.
\end{align}
We always use this dual (instead of the ungraded version), and usually, we just write $A$ again instead of $A^{\vee}$. Throughout, given $A, B\in \mathrm{End}(V)$, the bracket $[A, B] \in \mathrm{End}(V)$ denotes the graded commutator (or supercommutator)
\begin{equation*}
[A,B]:= AB-(-1)^{|A||B|}BA.
\end{equation*}
There are no ungraded commutators in this paper.

\section{$\vartheta$-summable Fredholm Modules over Locally Convex dg Algebras}\label{SectionFredholmModules}

Recall that a differential graded algebra $\Omega$ (or dg algebra for short) is an algebra that is the direct sum of subspaces $\Omega^j\subset \Omega$, $j \in \Z$, such that $\Omega^i\Omega^j\subset \Omega^{i+j}$ for all $i, j\in\Z$, together with a degree $+1$ differential $d$ which satisfies the graded Leibniz rule. We always assume that $\Omega$ has a unit $\mathbf{1} \in \Omega$. 

A {\em locally convex dg algebra} is a dg algebra $\Omega$, where the underlying vector space carries the structure of a locally convex vector space such that the differential is continuous and such that the product map is jointly continuous. Explicitly, these two conditions mean that for each continuous seminorm $\nu$ on $\Omega$, there exists a continuous seminorm $\nu'$ on $\Omega$ such that
\begin{equation} \label{EstimatesLocallyConvexAlgebra}
  \nu(d\theta)  \leq \nu'(\theta), \quad \nu(\theta_1\theta_2) \leq  \nu'(\theta_1)\nu'(\theta_2)\quad\text{  for all $\theta, \theta_1, \theta_2 \in \Omega$}.
\end{equation}
Moreover, we require that $\Omega$ is the {\em topological} direct sum of its homogenous summands $\Omega^j$; in particular, each of these is a closed subspace of $\Omega$.

\begin{definition}[$\vartheta$-summable Fredholm module]\label{fred}
An {\em (even) $\vartheta$-summable Fredholm module} over a locally convex dg algebra $\Omega$ is a triple $\Mod = (\Hil, \cc, Q)$, where
\begin{enumerate}[(i)]
\item $\Hil $ is a $\Z_2$-graded Hilbert space;
\item $\cc: \Omega \rightarrow \Lin(\Hil)$ (where $\Lin(\Hil)$ is the algebra of bounded operators on $\Hil$) is a parity-preserving, bounded linear map such that $\cc(\mathbf{1}) = 1$, the identity operator;
\item $Q$ is an odd self-adjoint unbounded operator on $\Hil$,
\end{enumerate}
such that we have
\begin{equation} \label{Multiplicativity}
  [Q, \cc(f)] = \cc(df), ~~~~\text{and}~~~~ \cc(f\theta) = \cc(f)\cc(\theta), ~~~\cc(\theta f) = \cc(\theta)\cc(f)
\end{equation}
for all $f \in \Omega^0$ and all $\theta \in \Omega$.
Moreover, we have the following analytic requirements.
\begin{enumerate}
\item[(A1)] For each $\theta\in \Omega$ the operators $C_\pm(\theta) := \Delta^{\pm1/2}\cc(\theta)\Delta^{\mp1/2}$ are densely defined and bounded, where $\Delta = Q^2+1$; moreover, the assignment $\theta \mapsto C_\pm(\theta)$ is bounded from $\Omega$ to $\Lin(\Hil)$;
\item[(A2)] For each $T>0$ the operator $e^{-T Q^2}\in \Lin(\Hil)$ is trace-class.
\end{enumerate}
If $\Mod$ satisfies (A1) and (A2) but not \eqref{Multiplicativity}, we call $\Mod$ a {\em  $\vartheta$-summable weak Fredholm module}.
\end{definition}

In the remainder of this section we give several examples of Fredholm modules. The following example highlights the relation to the usual notion of a $\vartheta$-summable Fredholm module over an algebra.

\begin{example}[Non-commutative differential forms] \label{ExampleNoncommutativeForms}
Let $\A$ be an ungraded locally convex algebra, together with a $\vartheta$-summable Fredholm module over $\A$ in the sense of Connes \cite{ConnesThetaSummable}. That is, we are given a triple $\Mod_0 = (\Hil, \cc_0, Q)$, where $\Hil$ is a $\Z_2$-graded Hilbert space, $Q$ is an odd selfadjoint unbounded operator satisfying (A2) and $\cc_0: \A \rightarrow \Lin^+(\Hil)$ a representation such that there exists a continuous seminorm $\nu$ on $\A$ with
\begin{equation} \label{EstimateNonDGAFredholmModule}
  \bigl\|\cc_0(a)\bigr\| +\bigl \|[Q, \cc_0(a)]\bigr\| \leq \nu(a)
\end{equation}
for all $a \in \A$. Under the additional assumption that there exists a continuous seminorm $\nu^\prime$ on $\A$ such that
\begin{equation} \label{AdditionalAssumption}
  \bigl\| \Delta^{-1/2} [Q, \cc_0(a)] \Delta^{1/2}\bigr\| + \bigl\| \Delta^{-1/2} [Q, \cc_0(a)^*] \Delta^{1/2}\bigr\| \leq \nu^\prime(a),
\end{equation}
where $\Delta = Q^2+1$, there exists a canonical extension of $\Mod_0$ to a Fredholm module $\Mod = (\Hil, \cc, Q)$ over the differential graded algebra of non-commutative differential forms $\Omega_\A$, as we now explain.

Let us start by reviewing the construction of $\Omega_\A$. First let $\Omega_\A^0 := \A$ and let $\Omega^1_\A$ be the bimodule of {\em one-forms on} $\A$, defined as the quotient of the free bimodule on generators $da$, $a \in \A$, by the relations
\begin{equation*}
\begin{aligned}
   d(\lambda a + \mu b) &= \lambda da + \mu db\\
   d(ab) &= a\, db + da\,b,
\end{aligned} 
~~~~~a, b \in \A, ~~\lambda, \mu \in \C.
\end{equation*}
Now we set 
\begin{equation*}
\Omega^n_\A := \underbrace{\Omega_\A^1 \otimes_\A \cdots \otimes_\A \Omega_\A^1}_{n} ~~~~~ \text{and} ~~~~~ \Omega_\A := \bigoplus_{n=0}^\infty \Omega_\A^n.
\end{equation*}
The product in $\Omega_\A$ is the obvious one, induced by the tensor algebra, while the formula
\begin{equation*}
  d(a_0 da_1 \cdots da_n) = da_0 da_1 \cdots da_n
\end{equation*}
gives is a well-defined differential which turns $\Omega_\A$ into a dg algebra. Here we use that using the second relation above, any element $\theta \in \Omega_\A^n$ can be written as a sum of elements of the form $a_0 da_1 \cdots da_n$ for $a_0, \dots, a_n \in \A$  (for a reference, see e.g., \cite[2.6.4]{Loday}). 

Now it can be checked that setting
\begin{equation} \label{ExtensionOfc0}
  \cc(a_0 da_1 \cdots da_n) := \cc_0(a_0) [Q, \cc_0(a_1)] \cdots [Q, \cc_0(a_n)]
\end{equation}
is a well-defined linear map $\Omega_\A \rightarrow \Lin(\Hil)$ which is obviously grading preserving (as $Q$ is odd) and satisfies \eqref{Multiplicativity}. Clearly, the condition (A1) follows from \eqref{EstimateNonDGAFredholmModule} and (A2) follows since
\begin{equation*}
\begin{aligned}
 \Delta^{1/2} \cc(a_0 da_1 \cdots da_n) \Delta^{-1/2} &= (\Delta^{1/2} \cc_0(a_0)\Delta^{-1/2})\prod_{k=1}^n(\Delta^{1/2} [Q, \cc_0(a_k)]\Delta^{-1/2}),
\end{aligned}
\end{equation*}
which can be estimated using \eqref{EstimateNonDGAFredholmModule} respectively \eqref{AdditionalAssumption}.
\end{example}

The following example is the main motivation for this paper.

\begin{example}[Differential forms and spinors] \label{ExampleSpinors}
  Let $X$ be an even-dimensional compact spin (or spin$^c$) manifold and let $\Sigma\to X$ be the corresponding complex spinor bundle with the Dirac operator $\DD$. There is a standard way to construct a $\vartheta$-summable Fredholm module $\Mod^X = (\Hil, Q, \cc)$ over $\Omega(X)$, the dg algebra of differential forms on $X$, a nuclear Fr\'echet algebra, due to compactness of $X$. Here $\Hil = L^2(X, \Sigma)$, the space of square-integrable sections of $\Sigma$ over $X$ and $Q = \DD$, the Dirac operator, while $\cc$ is the so-called {\em quantization map}  (c.f.\ \cite[Prop.~3.5]{BGV}): For differential one forms $\theta_1, \dots, \theta_k \in \Omega^1(X)$, the endomorphism of the spinor bundle $\cc(\theta_1 \wedge \cdots \wedge \theta_k)$ is given by
  \begin{equation} \label{QuantizationMapEx}
    \cc(\theta_1 \wedge \cdots \wedge \theta_k) = \frac{1}{k!} \sum_{\sigma \in S_k} \mathrm{sgn}(\sigma) \cc_1(\theta_{\sigma_1}) \cdots \cc_1(\theta_{\sigma_k}),
  \end{equation}
  where on the right hand side, $\cc_1(\theta_j)$ denotes Clifford multiplication by the vector field which corresponds to $\theta_j$ via the Riemannian structure on $X$. Each $\cc(\theta_1 \wedge \cdots \wedge \theta_k)$ acts as multiplication operator on $L^2(X, \Sigma)$. The properties \eqref{Multiplicativity} are well-known to hold in this case. Moreover, (A2) is satisfied because $\DD^2$ is a non-negative elliptic differential operator which thus has a smooth heat kernel; (A1) follows from elliptic estimates, as $[\DD, \cc(\theta)]$ is a first order differential operator for each $\theta$. Note that 
\begin{equation*}
\cc:\Omega(X)\longrightarrow \Lin\bigl(L^2(X, \Sigma)\bigr)
\end{equation*}
is \emph{not} multiplicative (in particular, not a representation), if $\dim(X) \geq 2$. 

This example has obvious generalizations to general unitary Clifford modules in the sense of \cite[Def.~3.32]{BGV}.
\end{example}

\begin{example}[Spin manifolds over $X$] \label{ExampleSpinors2}
Generalizing the previous example, if $X$ is any manifold and $E$ is a vector bundle with connection over $X$, then compact spin$^c$ manifolds $M$ over $X$ give rise to a $\vartheta$-summable Fredholm module $\Mod  = (\Hil, \cc, Q)$ over the locally convex algebra $\Omega(X)$. More precisely, suppose we are given an even-dimensional smooth spin$^c$ manifold $M$ with spinor bundle $\Sigma_M$, together with a smooth map $\varphi: M \rightarrow X$. Then we can set $\Hil = L^2(M, \Sigma_M\otimes \varphi^*E)$, $Q = \DD_\varphi$ (the Dirac operator on $M$, twisted by $\varphi^*E$) and $\cc(\theta) = \cc_M(\varphi^*\theta)$, where $\cc_M$ is the quanization map on $M$ as discussed in Example~\ref{ExampleSpinors}.
\end{example}

We now describe an important construction in this paper, which produces a new Fredholm module $\Mod_\T$ from a given Fredholm module $\Mod$; in fact, this will only be a {\em weak} Fredholm module, which is our main reason to considering this weaker notion in the first place. The motivation for this construction is its connection to equivariant homology on loop spaces in the special case of Example~\ref{ExampleSpinors}.

We will see in \S\ref{SectionBChOfP} that the Bismut-Chern character $\Ch(p)$ of an idempotent $p$ in $\Omega^0$ is not a chain over $\Omega$ but rather a chain over the acyclic extension $\Omega_{\T}$; hence we are going to need the Chern character of $\Mod_\T$ in order to formulate our index theorem.

\begin{example}[Acyclic extension of a Fredholm module] \label{ExampleExtension}
Given a locally convex dg algebra $\Omega$, we can form the dg algebra 
  \begin{equation*}
  \Omega_{\T} := \Omega[\sigma], 
  \end{equation*}
  where $\sigma$ is a formal variable of degree $-1$ with $\sigma^2 = 0$. Elements $\theta \in \Omega_{\T}$ can be written uniquely in the form $\theta = \theta^\prime + \sigma\theta^{\prime\prime}$, where $\theta^{\prime}, \theta^{\prime\prime}\in \Omega$, and the differential is
  \begin{equation} \label{DifferentialOmegaT}
   d_{\T} = d - \iota ~~~~~~~\text{with}~~~~~~~ d \theta = d\theta^\prime - \sigma d \theta^{\prime\prime} ~~~~ \text{and} ~~~~ \iota\theta = \theta^{\prime\prime}.
  \end{equation}
Then $\Omega_{\T}$ becomes a locally convex dg algebra in view of the vector space isomorphism $\Omega_{\T}\cong \Omega\oplus \Omega[1]$, which we call the {\em acyclic extension} of $\Omega$. $\Omega_\T$ is indeed acyclic, as all closed elements have the form $\theta^\prime + \sigma d \theta^\prime = - d_\T (\sigma \theta^\prime)$, hence are exact. In fact, up to sign choices, $\Omega_\T$ is in fact the mapping cone corresponding to the identity map $\Omega \rightarrow \Omega$.

Now given a  $\vartheta$-summable (weak) Fredholm module $\Mod = (\Hil, \cc, Q)$, the {\em acyclic extension of} $\Mod$ is the  $\vartheta$-summable weak Fredholm module $\Mod_\T = (\Hil, \cc_\T, Q)$ over $\Omega_{\T}$, where $\cc_\T$ is given by
\begin{equation*}
   \cc_{\T}(\theta^\prime + \sigma \theta^{\prime\prime}) := \cc(\theta^\prime).
\end{equation*}
We will often write $\cc$ again instead of $\cc_\T$. Notice that usually, this is indeed only a weak Fredholm module, since $(\Omega_\T)^0$ is strictly larger than $\Omega^0$ unless $\Omega^1 = \{0\}$.
\end{example}

\begin{remark} \label{RemarkAcyclicAndCircle}
  In the special case that $\Omega = \Omega(X)$, sending $\sigma$ to $dt$ gives an isomorphism of algebras 
  \begin{equation} \label{IsoOmegaT}
  \Omega_\T \cong \Omega(X \times S^1)^{S^1},
  \end{equation}
  the algebra of differential forms on $X \times S^1$ that are invariant under the $S^1$-action on the second factor. Under this isomorphsim, the differential $\iota$ on $\Omega_\T$ corresponds to the insertion of the vector field $\partial_t$ on $S^1$. Note that the isomorphism \eqref{IsoOmegaT} does not preserve degrees, as we declared $\sigma$ to have degree $-1$ in order to make the differential $d_\T$ homogeneous. Alternatively, this can be fixed by introducing a formal variable of degree two, but we shall avoid that in this paper.
\end{remark}

\section{Algebraic Preliminaries}\label{SectionPreliminaries}

In this section, we discuss the algebraic preliminaries needed for the construction of our Chern character. In particular, we introduce the chain and cochain complexes that we will be working on: The bar complex and the cyclic complex, as well as its entire and Chen normalized versions. Throughout this section, let $\Omega$ be a dg algebra.

\subsection{Bar Chains and Cochains} \label{BarChains}

The {\em bar construction} on $\Omega$ is the graded vector space 
\begin{equation} \label{FormulaBOmega}
   \B(\Omega)  = \bigoplus_{N=0}^\infty \Omega[1]^{\otimes N}.
\end{equation}
Here $\Omega[1]$ denotes the $\Z$-graded space with $\Omega[1]^k = \Omega^{k+1}$. The grading on $\B(\Omega)$ is then explicitly given by 
\begin{equation*}
  \B_n(\Omega) = \bigoplus_{N=0}^\infty \bigoplus_{\ell_1 + \dots + \ell_N = n+N} \Omega^{\ell_1} \otimes \cdots \otimes \Omega^{\ell_N}.
\end{equation*}
Elements of $\B(\Omega)$ will be called {\em bar chains on} $\Omega$ and denoted by $(\theta_1, \dots, \theta_N)$, omitting the tensor product sign in notation.
The space $\B(\Omega)$ carries the two differentials $d$ and $b^\prime$, given by\footnote{The sign of $b^\prime$ coincides with that of Quillen's $b'$ \cite{Quillen}, who has considered the case of ungraded algebras.}
\begin{equation*}
\begin{aligned}
d(\theta_1, \dots, \theta_N) &= \sum_{k=1}^N (-1)^{n_{k-1}}({\theta}_1, \dots, {\theta}_{k-1}, d \theta_k, \dots, \theta_N)\\
b^\prime(\theta_1, \dots, \theta_N) &= -\sum_{k=1}^{N-1} (-1)^{n_k}({\theta}_1, \dots, {\theta}_{k-1}, {\theta}_{k}\theta_{k+1}, \theta_{k+2}, \dots, \theta_N)
\end{aligned}
\end{equation*}
where $n_k = |\theta_1|+\dots + |\theta_k|-k$. Observe that both differentials have degree one; $d$ preserves the arity, while $b^\prime$ decreases it. We have $db^\prime + b^\prime d = 0$, hence these differentials turn $\B(\Omega)$ into a bi-complex with total differential $d + b^\prime$. 

\medskip

Given a $\mathbb{Z}_2$-graded algebra $L$ of coefficients, we define an \emph{$L$-valued bar cochain on $\Omega$} to be an element of the space $\Hom(\B(\Omega), L)$ of linear maps from $\B(\Omega)$ to $L$.
Such an $L$-valued bar cochain can be identified with a sequence $\ell^{(0)}, \ell^{(1)}, \ell^{(2)}, \dots$, consisting of its arity $N$ components; each component is a multi-linear map 
\begin{equation*}
  \ell^{(N)}:  \underbrace{\Omega  \times \dots \times  \Omega}_{N} \longrightarrow L. 
\end{equation*}
In particular, the component $\ell^{(0)}$ is a linear map from $\C$ to $L$, which can be identified with an element of $L$.

\medskip

Following Quillen \cite{Quillen} we introduce a sign in the definition of the {\em codifferential}
\begin{equation}
\delta:=-(d+b'): \Hom\bigl(\B(\Omega), L\bigr)\longrightarrow \Hom\bigl(\B(\Omega), L\bigr),
\end{equation}
which is defined as in \eqref{dualy}. It will be important for our constructions that the space of $L$-valued bar cochains is a $\Z_2$-graded algebra with respect to the product given by 
\begin{equation} \label{FormulaForProduct}
  \ell_1\ell_2 (\theta_{1}, \dots, \theta_N) = \sum_{k=1}^N (-1)^{|\ell_2|(|\theta_{1}| + \dots + |\theta_k|)}\ell_1(\theta_1, \dots, \theta_{k})\ell_2(\theta_{k+1}, \dots, \theta_N),
\end{equation}
in a way that $\delta$ satisfies the $\Z_2$-graded Leibniz rule
\begin{equation*}
  \delta(\ell_1 \ell_2) = (\delta \ell_1) \ell_2 + (-1)^{|\ell_1|} \ell_1 (\delta \ell_2).
\end{equation*}
Thus $\delta$ turns $\Hom(\B(\Omega), L)$ into a differential $\Z_2$-graded algebra.

\medskip

As observed by Quillen \cite{Quillen}, an odd bar cochain $\omega$ with values in a $\Z_2$-graded algebra $L$ of coefficients can be seen as a {\em connection form} on the space of $L$-valued bar cochains, as we can consider the {\em connection}\footnote{Since $\omega$ does not necessarily have arity one, this should rather be called a {\em superconnection} in the sense of Quillen.} $\nabla: = \delta + \omega$. Its {\em curvature} is given by
\begin{equation} \label{CurvatureDefinition}
  F: = \nabla^2 = \delta \omega + \omega^2.
\end{equation}
Note that since $\delta^2 = 0$, $F$ is an $L$-valued bar cochain itself, as opposed to merely an operator acting on cochains. It satisfies the {\em Bianchi identity}
\begin{equation} \label{BianchiIdentity}
  0=[\nabla, F] = \delta F + [\omega, F],
\end{equation}
since $[\nabla, F] = \nabla^3 - \nabla^3 = 0$. These two formulas lie at the heart of our investigations.

\subsection{Cyclic Chains and Cochains}

Throughout, we denote by $\underline{\Omega}$ the quotient space $\underline{\Omega} := \Omega / \C\mathbf{1}$.
The {\em (reduced) cyclic complex} of Connes associated to $\Omega$ is the graded vector space 
\begin{equation} \label{FormulaCOmega}
   \mathsf{C}(\Omega) = \bigoplus_{N=0}^\infty \Omega \otimes \underline{\Omega}[1]^{\otimes N},
\end{equation}
where we remind the reader that $\underline{\Omega}[1]$ denotes the graded vector space with $\underline{\Omega}[1]^k = \underline{\Omega}^{k+1}$. Explicitly, the grading of $\CC(\Omega)$ is therefore  given by
\begin{equation*}
  \mathsf{C}_n(\Omega) = \bigoplus_{M=0}^\infty \bigoplus_{\ell_0 + \dots + \ell_M = n+M} \Omega^{\ell_0}\otimes \underline{\Omega}^{\ell_1} \otimes \cdots \otimes \underline{\Omega}^{\ell_M}.
\end{equation*}
Elements of $\CC(\Omega)$ will be called {\em cyclic chains}, denoted by $(\theta_0, \dots, \theta_N)$.
The complex $\CC(\Omega)$ carries two degree $+1$ differentials, given by the formulas
\begin{equation*}
\begin{aligned}
  \underline{d}(\theta_0, \dots, \theta_N) &= (d\theta_0, \theta_1, \dots, \theta_N) - \sum_{k=1}^N (-1)^{m_{k-1}}({\theta}_0, \dots {\theta}_{k-1}, d\theta_k, \theta_{k+1}, \dots, \theta_N)\\
  \underline{b}(\theta_0, \dots, \theta_N) &= \sum_{k=0}^{N-1} (-1)^{m_k}({\theta}_0, \dots, {\theta}_k\theta_{k+1}, \dots, \theta_N) - (-1)^{(|\theta_N|-1)m_{N-1}}(\theta_N\theta_0, \theta_1, \dots, \theta_{N-1}).
  \end{aligned}
\end{equation*}
where $m_k = |\theta_0|+\dots + |\theta_k|-k$. 
It is easy to check that these formulas make sense on the quotient space $\Omega \otimes \underline{\Omega}[1]^{\otimes N}$. There is another differential, of degree $-1$, the {\em Connes operator} $\underline{B}$, which is given by the formula
\begin{equation*}
  \underline{B}(\theta_0, \dots, \theta_N) = \sum_{k=0}^N (-1)^{(m_{k}+1)m_N}(\mathbf{1}, \theta_{k+1}, \dots, \theta_N, \theta_0, \dots, \theta_{k})
\end{equation*}
with $m_k$ as before. We remark that the Connes differential is one reason to restrict discussion to the reduced case, since its formula becomes more complicated otherwise, c.f.\ e.g. \cite[p.~13]{GJP}. 
The three differentials $\underline{d}$, $\underline{b}$ and $\underline{B}$ pairwise anti-commute, hence we get a $\Z_2$-graded complex
\begin{equation} \label{cycl}
\begin{tikzcd}
\CC_{+}(\Omega) \ar[rr,  bend left=20, "\underline{d}+\underline{b}+\underline{B}"] & &  \ar[ll,  bend left=20, "\underline{d}+\underline{b}+\underline{B}"] \CC_{-}(\Omega),
\end{tikzcd}
%\qquad 
%\begin{tikzcd} \mathrm{Hom}_+(\mathsf{C}(\Omega), \C) \ar[r,  bend left=20, "\underline{d}+\underline{b}+\underline{B}"] &  \ar[l,  bend left=20, "\underline{d}+\underline{b}+\underline{B}"]  \mathrm{Hom}_-(\mathsf{C}(\Omega), \C),
%\end{tikzcd}
\end{equation}
where the spaces $\CC_\pm(\Omega)$ are given by reducing the $\Z$-grading of $\CC(\Omega)$ modulo two. Note that this is not a $\Z$-graded complex, as the differential is inhomogeneous.

\subsection{Analytic Cochains and Entire Chains}\label{AnaEnt}

In this section, let $\Omega$ be a locally convex dg algebra. As for bar cochains, a complex-valued cochain $\ell$ on $\CC(\Omega)$ can be identified with the sequence $\ell^{(0)}, \ell^{(1)}, \ell^{(2)}, \dots$ of its arity $N$ components, each component being a multi-linear map 
\begin{equation*}
  \ell^{(N)}:  \Omega\times \underbrace{\underline{\Omega}[1]  \times \dots \times  \underline{\Omega}[1]}_{N} \longrightarrow \C. 
\end{equation*}
Following the nomenclature of \cite[\S2.1.2]{MeyerBook}, we introduce the following growth condition on cochains.

\begin{definition}[Analytic cochains] \label{DefAnalytic}
A cochain $\ell \in \Hom(\mathsf{C}(\Omega), \C)$ is {\em analytic} if there exists a continuous seminorm $\nu$ on $\Omega$ and a constant $C>0$, such that\footnote{Here $\lfloor N/2\rfloor = N/2$ if $N$ is even, and $\lfloor N/2\rfloor = (N-1)/2$ if $N$ is odd.}
\begin{equation} \label{GrowthConditionAnalytic}
  \bigl|\ell^{(N)}(\theta_0, \dots, \theta_N)\bigr| \leq \frac{C}{\lfloor N/2\rfloor !}\nu(\theta_0) \cdots \nu(\theta_N)
\end{equation}
for all $N \in \N$ and $\theta_0, \dots, \theta_N \in \Omega$. The space of analytic cochains will be denoted by $\CC_\alpha(\Omega)$.
\end{definition}

In particular, each of the linear maps $\ell^{(N)}$ is continuous with respect to the projective tensor product topology on $\Omega \otimes \underline{\Omega}^{\otimes N}$, see below. The differentials $\underline{d}$, $\underline{b}$ and $\underline{B}$ are easily checked to preserve the space $\CC_\alpha(\Omega)$, leading to the {\em analytic cocomplex} of $\Omega$.

\medskip

The growth conditions imposed the components $\ell^{(N)}$ of analytic cochains $\ell$ allow to evaluate such a cochain on certain chains that are an infinite sum of their homogenous tensor components, in other words, element of the direct {\em product} of the $\Omega \otimes \Omega[1]^{\otimes N}$. To define these {\em entire chains}, we recall that for each continuous seminorm $\nu$ on $\Omega$ the induced {\em $N$-th projective tensor norm} $\pi_{\nu, N}$ on $\Omega \otimes \Omega[1]^{\otimes N}$ is defined by the formula
\begin{equation} \label{AssociatedProjectiveNorm}
  \pi_{\nu, N}(c_N) = \inf  \left\{ \sum_\gamma \nu\big(\theta_0^{(\gamma)}\big) \cdots \nu\big(\theta_N^{(\gamma)}\big) : c_N = \sum_\gamma \theta_0^{(\gamma)}\otimes \cdots \otimes\theta_N^{(\gamma)}   \right\},
\end{equation}
where the infimum is taken over all representations of $c_N \in \CC(\Omega)$ as a finite sum of elementary tensors. 

\begin{definition}[Entire chains] \label{DefEntireComplex}
For each continuous seminorm $\nu$ on $\Omega$, we define a seminorm $\epsilon_\nu$ on $\CC(\Omega)$ by
\begin{equation} \label{EntireNorm}
 \epsilon_\nu(c) :=  \sum_{N=0}^\infty \frac{\pi_{\nu, N}(c_N)}{ \lfloor N/2\rfloor!}, \qquad c = \sum_{N=0}^\infty c_N \in \CC(\Omega), ~\text{with}~ c \in \Omega \otimes \Omega[1]^{\otimes N},
\end{equation}
which is finite as the definition of $\CC(\Omega)$ involves the direct sum.
The completion of $\CC(\Omega)$ with respect to the seminorms $\epsilon_\nu$ is denoted by $\CC^\epsilon(\Omega)$. Its elements are called {\em entire chains}.
\end{definition}

It is easy to see that analytic cochains $\ell \in \CC_\alpha(\Omega)$ are bounded with respect to the seminorms $\epsilon_\nu$, hence they extend by continuity to functionals on $\CC^\epsilon(\Omega)$. Moreover, as the differentials $\underline{d}$, $\underline{b}$, $\underline{B}$ are easily checked to be continuous  on $\mathsf{C}(\Omega)$ with respect to the seminorms $\epsilon_\nu$, they extend to odd parity differentials on $\CC^\epsilon(\Omega)$; we denote these extensions by the same symbols again, leading to the \emph{entire complex} of $\Omega$. 

\begin{remark}
The terminology {\em analytic} for the growth condition \eqref{GrowthConditionAnalytic} stems from the fact that this condition implies that the function
\begin{equation*}
  f(\theta) = \sum_{N=0}^\infty \lfloor N/2\rfloor! \,\ell^{(N)}(\underbrace{\theta, \dots, \theta}_{N+1})
\end{equation*}
defines an analytic function on some small enough $\nu$-ball around zero in $\Omega$. Similarly, for any entire chain $c$ and any continuous seminorm $\nu$ on $\Omega$, the function
\begin{equation*}
  f(z) := \sum_{N=0}^\infty \pi_{\nu, N}(c_N) \frac{z^N}{\lfloor N/2\rfloor!}
\end{equation*} 
is an entire function of $z$. 
This follows from \eqref{EntireNorm}, after replacing the seminorm $\nu$ by the seminorm $|z|\nu$.
The entire growth condition was first introduced by Connes \cite{ConnesEntire} for {\em cochains} on the space $\CC(\mathcal{A})$ for a Banach algebra $\mathcal{A}$; it then turns out \cite[\S2.1.2]{MeyerBook} that the space of entire cochains is the dual to the space of chains satisfying the analytic growth condition. 

It turns out that in our context, it is suitable to consider the entire growth condition on chains rather on cochains, and the analytic growth condition on cochains, rather than on chains, due to the fact that the cochain Chern character $\Ch_\Mod$ satisfies the analytic growth condition (c.f.\ Thm.~\ref{ThmFundamentalEstimate}), while the Chern character defined in Def.~\ref{DefBChP} satisfies the entire growth condition, but not the analytic one.

For the algebra of differential forms on a manifold, these growth conditions on chains were previously considered e.g., in \cite{LeandreFrench} and \cite{JonesLeandre}; see also \cite{GetzlerSzenes}.
\end{remark}

\subsection{Chen Normalization}\label{ChenNorm}

Again, we assume that $\Omega$ is a locally convex dg algebra with acyclic extension $\Omega_\T$, discussed in Example~\ref{ExampleExtension}. In this section, we construct the normalized complexes that our constructions will ultimately live on. To this end, we consider a certain subcomplex $\mathsf{D}_\T(\Omega)$ of $\CC(\Omega_\T)$.

To define the subcomplex $\DD^\T(\Omega)$ of $\CC(\Omega_\T)$ consider for all $f \in \Omega^0$ and $i \in \N_0$ the continuous linear map
\begin{equation*}
S_i^{(f)}:\CC(\Omega_\T)\longrightarrow \CC(\Omega_\T)
\end{equation*}
defined by
\begin{equation} \label{DefinitionSif}
  S_i^{(f)}(\theta_0, \dots, \theta_N) = \begin{cases} (\theta_0, \dots, \theta_{i}, f, \theta_{i+1}, \dots, \theta_N) & 0 \leq i \leq N \\ 0 & \text{otherwise}\end{cases}.
\end{equation}
We emphasize here that we only consider $f \in \Omega^0 \subset \Omega_\T^0$, which is generally a proper subset. Consider further the maps
\begin{align} \label{DefinitionOfOperatorS}
  &S: \CC(\Omega_\T) \longrightarrow \CC(\Omega_\T), & &S(\theta_0, \dots, \theta_N) = \sum_{k=0}^N (\theta_0, \dots, \theta_{k}, \sigma, \theta_{k+1}, \dots, \theta_N),\\
\label{DefinitionOfOperatorR}
  &R: \CC(\Omega_\T) \longrightarrow \CC(\Omega_\T), & &R(\theta_0, \dots, \theta_N) = (\sigma \theta_0, \dots, \theta_N).
\end{align}
We then define the subcomplex $\DD^\T(\Omega) \subset \CC(\Omega_\T)$ by
\begin{equation} \label{DefinitionDT}
  \DD^\T(\Omega) := \left\{ 
  \begin{aligned}&\text{The span of the images of the operators} \\ &\text{$S$, $R$ and $S^{(f)}_i$, $[\underline{d} + \underline{b}, S_i^{(f)}]$ for $f \in \Omega^0$} 
  \end{aligned}
  \right\}.
\end{equation}
It is straightforward to check that the differentials $\underline{d}$, $\underline{b}$ and $\underline{B}$ preserve this complex.

\begin{definition}[The Chen normalized complex] \label{DefExtendedEntireNormalized}
The {\em Chen normalized entire complex} is the quotient complex
\begin{equation} \label{NormalizedComplex}
   \NN^{\T, \epsilon}(\Omega) := \CC^{\epsilon}(\Omega_\T) / \overline{\DD^{\T}(\Omega)},
\end{equation}
where we take the closure inside $\CC^\epsilon(\Omega_\T)$. The space of {\em extended Chen normalized analytic cochains} is the space $\NN_{\T, \alpha}(\Omega)$ of analytic cochains vanishing on $\DD^\T(\Omega)$. We obtain the short complexes
\begin{equation*}
\begin{tikzcd}
\NN_{+}^{\T, \epsilon}(\Omega) \ar[rr,  bend left=20, "\underline{d}+\underline{b}+\underline{B}"] & &  \ar[ll,  bend left=20, "\underline{d}+\underline{b}+\underline{B}"] \NN_{-}^{\T, \epsilon}(\Omega),
\end{tikzcd}
~~~~\text{and}~~~~
\begin{tikzcd}
\NN_{\T, \alpha}^+(\Omega) \ar[rr,  bend left=20, "\underline{d}+\underline{b}+\underline{B}"] & &  \ar[ll,  bend left=20, "\underline{d}+\underline{b}+\underline{B}"] \NN_{\T, \alpha}^-(\Omega).
\end{tikzcd}
\end{equation*}
\end{definition}

\begin{remark}
The subcomplex $\DD(\Omega)$ of $\CC(\Omega)$, generated by the images of the operators $S_i^{(f)}$ and $[\underline{d} + \underline{b}, S_i^{(f)}]$, was first considered by Chen \cite{Chen2}, in the case that $\Omega = \Omega(X)$ for a manifold $X$. He also showed that $\DD(\Omega(X))$ is acyclic with respect to the differential $\underline{d}+\underline{b}$ in the case that $X$ is connected, meaning that the projection from $\CC(\Omega(X))$ to the quotient complex 
\begin{equation*}
\NN\bigl(\Omega(X)\bigr) := \CC\bigl(\Omega(X)\bigr)\big/\DD\bigl(\Omega(X)\bigr)
\end{equation*} 
is a homotopy equivalence in this case (with respect to the differential $\underline{d}+\underline{b}$). This was suitably generalized by Getzler and Jones to general dg algebras \cite[Section~5]{GetzlerJones}.

This quotient \eqref{NormalizedComplex} was first considered by Getzler, Jones and Petrack in the special case $\Omega = \Omega(X)$ for a manifold $X$; in fact, the image of $S-\mathrm{id}$ can be written as an ideal with respect to the shuffle product, as defined in \S4 of \cite{GJP}; it is generated by the chain $(1, \sigma) - (1) \in \CC(\Omega)$. Cast in these terms, such a quotient is considered on p.~28 of \cite{GJP}. Clearly, there is a map of chain complexes 
\begin{equation*}
\NN^\epsilon(\Omega) := \CC^\epsilon(\Omega) / \overline{\DD(\Omega)}\longrightarrow \NN^{\T, \epsilon}(\Omega), 
\end{equation*}
natural in $\Omega$. The results of Getzler, Jones and Petrack seem to indicate that one might suspect this map to be a quasi-isomorphism, at least in the special case $\Omega = \Omega(X)$.
\end{remark}

\section{The Quantization Map} \label{SectionQuantizationMap}

In this section, we start with the construction of the Chern character for a  $\vartheta$-summable weak Fredholm module $\Mod:=(\Hil, \cc, Q)$ over a locally convex dg algebra, which we assume fixed throughout this section.

The Fredholm module $\Mod$ naturally provides a connection form in the sense of Section~\ref{BarChains}, i.e., and odd bar cochain $\omega_{\Mod}$ with values in the algebra of linear operators on $\Hil$. It is defined by
\begin{equation} \label{ComponentsOfOmega}
\omega^{(0)}_{\Mod} = -Q, ~~~~~~ \omega^{(1)}_{\Mod}(\theta) = \cc(\theta),~~~~~~~\omega^{(N)}_{\Mod}(\theta_1, \dots, \theta_N) = 0, ~~N\geq 2,
\end{equation}
and it is odd due to the grading shift in the definition of $\B(\Omega)$. 
Let 
\begin{equation} \label{CurvatureOfOmegaMod}
F_{\Mod} =  \delta \omega_{\Mod} + \omega_{\Mod}^2
\end{equation}
 be the curvature of the connection $\omega$, as defined in \eqref{CurvatureDefinition}. 
 Explicity, the components of $F_{\Mod}$ can be easily worked out to be given by 
\begin{equation} \label{ComponentsOfF}
\begin{aligned}
  &F^{(0)}_{\Mod} & &\!\!\!\!\!= Q^2,\\
  &F^{(1)}_{\Mod}(\theta)& &\!\!\!\!\!= \cc(d \theta) - [Q, \cc(\theta)]\\
  &F^{(2)}_{\Mod}(\theta_1, \theta_2)& &\!\!\!\!\!= (-1)^{|\theta_1|}\bigl(\cc({\theta}_1{\theta}_2) - \cc({\theta}_1)\cc({\theta}_2) \bigr).
\end{aligned}
\end{equation}
with all higher components zero. Notice that the first component essentially is the curvature of the derivation on $\Lin(\Hil)$ given by taking the commutator with $Q$, the second component measures the failure of $\cc$ and $Q$ to satisfy the relation $\cc(d\theta) = [Q, \cc(\theta)]$ for $\theta \in \Omega$, while the third component measures the failure of $\cc$ to be multiplicative.
What causes complications is that $\omega_{\Mod}$ and $F_{\Mod}$ do not take values in bounded operators (hence does not actually take values in an algebra).
In particular, the identity \eqref{CurvatureOfOmegaMod} is only formal.
However, $F_{\Mod}$ is well-defined by \eqref{ComponentsOfF}.
For these reasons, Quillen's constructions are more of a guiding principle, and we have to put in some effort to make them work. 

In order to define the Chern character of our Fredholm module, we now aim to define a bar cochain $\Phi_T^\Mod$ with values in $\Lin(\Hil)$, the {\em quantization map}, by exponentiating the curvature $F_{\Mod}$ in the algebra $\Hom(\B(\Omega), \Lin(\Hil))$, i.e., we would like to set 
\begin{equation} \label{FormalEquality}
\Phi_T^\Mod \stackrel{\text{formally}}{=} e^{-TF_{\Mod}},
\end{equation}
where $T>0$ is any real parameter. If $Q$ was a bounded operator, this would be well-defined using the exponential series; however, the fact that $Q$ is unbounded leads to analytical difficulties. To overcome these difficulties, the plan is to split 
\begin{equation} \label{SplitUp}
F_{\Mod} = Q^2 + F^{\geq 1}_{\Mod},
\end{equation}
where $F^{\geq 1}_{\Mod}$ is the part containing the cochains of arity at least one. Using this, we have (formally) 
\begin{equation*}
e^{-TF_{\Mod}} = e^{-T(Q^2 + F^{\geq 1}_{\Mod})}
\end{equation*}
which can be rewritten into a perturbation series with respect to the heat operator $e^{-TQ^2}$. To facilitate this, we first introduce the following notation.

\begin{notation}[Bracket] \label{NotationBracket}
Let $H$ be a non-negative operator on a Hilbert space $\Hil$.
Given suitable\footnote{In the sense that they satisfy the assumptions of Lemma~\ref{LemmaMainEstimate} below.} operators $A_1, \dots, A_N$ on $\Hil$, we set
\begin{equation} \label{FormulaBracket}
   \bigl\{A_1, \dots, A_N\bigr\}_{H} := \int_{\Delta_{N}} e^{-\tau_1H} A_1 e^{-(\tau_2-\tau_1)H}A_2 \cdots e^{-(\tau_N-\tau_{N-1})H} A_N e^{-(1-\tau_N)H} \dd \tau
   % \prod_{j=1}^N A_j e^{-T(\tau_j-\tau_{j-1})Q^2} \dd \tau,
\end{equation}
where $\Delta_{N} = \{\tau \in \R^n \mid 0 \leq \tau_1 \leq \dots \leq \tau_N \leq 1 \}$ is the standard simplex. 
\end{notation}

In order to define $\Phi_T^\Mod$, we will take $H = Q^2$. In view of \eqref{ComponentsOfF}, we will have to allow unbounded operators $A_j$ to be put into the bracket above. The following fundamental estimate will facilitate this. 

\begin{lemma} \label{LemmaMainEstimate}
For $0 \leq a_0, \dots, a_N < 1$ let $A_0, \dots, A_N$ be densely defined operators on $\Hil$ such that $A_j(H+1)^{-a_j}$ is densely defined and extends to a bounded operator on $\Hil$. Then for each $T>0$, $A_0\{A_1, \dots, A_N\}_{TH}$ is a well-defined trace-class operator, and we have the estimate
  \begin{equation*}
    \bigl\| A_0\{A_1, \dots, A_N\}_{TH}\bigr\|_1 \leq \frac{e^{T/2} \Tr\bigl(e^{-TH/2}\bigr)}{\Gamma\bigl(N+1 - (a_0+\dots + a_N)\bigr)}
    \prod_{j=0}^N \bigl\|A_j(H+1)^{-a_j}\bigr\| \left(\frac{2a_j\Gamma(1-a_j)}{eT}\right)^{a_j},
  \end{equation*}
for its trace class norm, where $\Gamma$ denotes the gamma function.
\end{lemma}

In fact, in order to define $\Phi_T^{\Mod}$ as a continuous $\Lin(\Hil)$-valued cochain on $\B(\Omega)$, it would be sufficient to prove the above estimate for the operator norm instead of the trace class norm. 
However, since the definition of the Chern character in Section~\ref{SectionChernCharacter} involves taking a trace, we need the stronger estimate above.

\begin{proof}
By the assumptions on $A_j$, for each $\varepsilon >0$, the operator 
\begin{equation*}
A_j e^{-\varepsilon H} = (A_j (H+1)^{-a_j})\bigl((H+1)^{a_j} e^{-\varepsilon/2 H}\bigr) e^{-\varepsilon/2 H}
\end{equation*} 
is trace-class, hence for each fixed $\tau \in \Delta_N$ such that $\tau_j - \tau_{j-1} > 0$, the integrand in \eqref{FormulaBracket} is a trace-class operator, by the assumptions on the $A_j$. In order to show that the integral over $\Delta_N$ gives a well-defined trace-class operator, it then suffices to show that the trace class norm as a function of $\tau$ is integrable.

To this end, set $s_j = \tau_{j+1} - \tau_{j}$, $j=0, \dots, N$, where $\tau_{N+1} = 1$, $\tau_0 = 0$. Since $\sum_{j=0}^{N} s_j = 1$, we have using H\"older's inequality for Schatten norms
\begin{equation*}
\begin{aligned}
  \left\| \prod_{j=0}^N A_j e^{-T(\tau_j - \tau_{j-1})H}\right\|_1 &= \left\| \prod_{j=0}^N A_j e^{-Ts_jH}\right\|_1\leq \prod_{j=0}^N \bigl\|A_j (H+1)^{-a_j}\bigr\| \bigl\|(H+1)^{a_j} e^{-Ts_j H}\bigr\|_{1/s_j}\\
  &\leq \prod_{j=0}^N \bigl\|A_j (H+1)^{-a_j}\bigr\|\prod_{j=0}^N \bigl\|(H+1)^{a_j} e^{-Ts_jH/2}\bigr\| \prod_{j=0}^N\bigl\|e^{-Ts_jH/2}\bigr\|_{1/s_j}.
\end{aligned}
\end{equation*}
Now for the last product, we get
\begin{equation*}
  \prod_{j=0}^N\bigl\|e^{-Ts_jH/2}\bigr\|_{1/s_j} = \prod_{j=0}^N \left(\sum_{k=1}^\infty \bigl(e^{-Ts_j\lambda_k/2}\bigr)^{1/s_j}\right)^{s_j} = \prod_{j=0}^N \Tr\bigl(e^{-TH/2}\bigr)^{s_j} = \Tr(e^{-TH/2}),
\end{equation*}
where $\lambda_j$, for $j=1, 2, \dots$, are the eigenvalues of $H$. To estimate the second product, observe that 
\begin{equation*}
  (x^2+1)^a e^{-\tau x^2/2} \leq \left(\frac{2a}{\tau}\right)^a e^{\tau/2 - a}.
\end{equation*}
for all $x \in \R$. Hence
\begin{equation*}
  \prod_{j=0}^N \bigl\|(H+1)^{a_j} e^{-Ts_jH/2}\bigr\| \leq e^{T/2} \prod_{j=0}^N \left(\frac{2a_j}{eT}\right)^{a_j} (\tau_{j+1}-\tau_j)^{-a_j}
\end{equation*}
and the lemma now follows from the formula
\begin{equation*}
 \int_{\Delta_N}\tau_1^{-a_0}\left(\tau_2-\tau_1\right)^{-a_1}\dots \left(1-\tau_N\right)^{-a_{N}} d\tau = \frac{\Gamma(1-a_0)\cdots \Gamma(1-a_N)}{\Gamma\bigl(N+1-(a_0+\dots + a_N)\bigr)};
\end{equation*}
in the case $N=1$, this latter formula is the usual integral formula for the beta function, while the case for general $N$ can be easily shown by induction.
\end{proof}

As can be seen from \eqref{ComponentsOfF}, the operators $F_{\Mod}(\theta)$ are unbounded operators in general. However, by the assumption (A1) of the Fredholm module, they satisfy the assumptions of Lemma~\ref{LemmaMainEstimate}. Therefore, using the algebra structure on the space of $\Lin(\Hil)$-valued bar cochains, the expression $\{F^{\geq 1}_{\Mod}, \dots, F^{\geq 1}_{\Mod}\}_{TQ^2}$, used in the definition below, defines a well-defined cochain.

\begin{definition}[Quantization Map]\label{DefinitionPhi}
For $T>0$ fixed, the {\em quantization map} $\Phi^{\Mod}_T$ is the $\Lin(\Hil)$-valued bar cochain defined by the formula
\begin{equation} \label{PerturbationSeriesPhiT}
 \Phi^{\Mod}_T:= \sum_{N=0}^\infty (-T)^N \bigl\{\underbrace{F^{\geq 1}_{\Mod}, \dots, F^{\geq 1}_{\Mod}}_N\bigr\}_{TQ^2}.
\end{equation}
\end{definition}

\begin{remark} \label{RemarkPhiEven}
$\Phi^{\Mod}_T$ is an even cochain, in the sense that it maps even elements of $\B(\Omega)$ to even operators on $\Hil$ and odd elements to odd operators. This follows from the fact that $F$ has this property (which in turn holds because $\omega$ is odd, or can be seen from an inspection of \eqref{ComponentsOfF} above), and the fact that the bracket is also even in the obvious sense (as it depends only on the square of $Q$).
\end{remark}

To state explicitly how $\Phi_T^{\Mod}$ evaluates on chains, we need the following notation.

\begin{notation}[Partitions] \label{NotationPartitions}
Denote by $\mathscr{P}_{M, N}$ the set of ordered partitions of the set $\{1, \dots, N\}$ of length $M$, i.e., collections of nonempty subsets $I=(I_1, \dots, I_M)$ such that $I_1 \cup \dots \cup I_M = \{1 \dots, N\}$ and such that each element of $I_a$ is smaller than any element of $I_b$ whenever $a < b$.
\end{notation}

Now an inspection of the definition of $\Phi_T^\Mod$ gives the formula
\begin{equation} \label{ExplicitPhiT}
  \Phi^{\Mod}_T(\theta_1, \dots, \theta_N) = \sum_{M=1}^N (-T)^M \sum_{I \in \mathscr{P}_{M, N}} \bigl\{ F_{\Mod}(\theta_{I_1}), \dots, F_{\Mod}(\theta_{I_M})\bigr\}_{TQ^2}  ,
\end{equation}
where we used the obvious notation $\theta_{I_a}:= (\theta_{i+1} , \dots , \theta_{i+m})$
 if $I_a = \{j \mid i < j \leq i+m\}$ for some $i, m$. This formula follows from the product formula \eqref{FormulaForProduct} for cochains. Notice that a summand corresponding to $I = (I_1, \dots, I_M) \in \mathscr{P}_{M, N}$ is zero as soon as $I$ contains an $I_a$ with $|I_a| \geq 3$. 

The following are now the main results of this section. 

\begin{theorem}[Fundamental estimate] \label{ThmFundamentalEstimate}
For $T>0$ and all $\theta_1, \dots, \theta_N \in \Omega$, the operator $\Phi^{\Mod}_T(\theta_1, \dots, \theta_N)$ is well-defined and trace-class. Moreover, there exists a continuous seminorm $\nu$ on $\Omega$, independent of $T$, such that we have the estimate
\begin{equation} \label{TraceNormEstimatePhi}
  \bigl\|\Phi^{\Mod}_T(\theta_1, \dots, \theta_N)\bigr\|_1 \leq e^{T/2} \mathrm{Tr}(e^{-TQ^2/2})\frac{T^N}{\lfloor N/2\rfloor!} \nu(\theta_1) \cdots \nu(\theta_N)
\end{equation}
for its trace class norm. The same statement is true for $Q\Phi^{\Mod}_T$ and $\Phi^{\Mod}_T Q$ instead of $\Phi^{\Mod}_T$.
\end{theorem}

\begin{theorem}[Bianchi identity] \label{ThmBianchiPhi}
The map $\Phi_T^{\Mod}$ satisfies
\begin{equation*}
\delta \Phi^{\Mod}_T  + \bigl[\omega, \Phi^{\Mod}_T \bigr] = 0.
\end{equation*}
\end{theorem}

Thm.~\ref{ThmBianchiPhi} is obvious at a formal level, using the formal equality \eqref{FormalEquality} together with the Bianchi identity \eqref{BianchiIdentity} for $F_\Mod$, but again, the fact that $F_\Mod$ does not take values in $\Lin(\Hil)$ does not allow this easy proof.

\begin{remark}
Explicitly, on elements, this identity can be worked out to state
\begin{equation} \label{ExplicitThmBianchiPhi}
  \begin{aligned}
   \Phi^{\Mod}_{T}\bigl((d+b^\prime)(\theta_1, \dots, \theta_N)\bigr)  &=  \cc(\theta_1)\Phi^{\Mod}_{T}(\theta_{2}, \dots, \theta_N)- (-1)^{n_{N-1}}\Phi^{\Mod}_{T}(\theta_1, \dots, \theta_{N-1}) \cc(\theta_N)\\
   &~~~ -  [Q, \Phi^{\Mod}_{T}(\theta_1, \dots, \theta_N)] ,
  \end{aligned}
  \end{equation}
where $n_k = |\theta_1| + \dots +|\theta_k|-k$. Here both sides are well-defined by Thm.~\ref{ThmFundamentalEstimate}.
\end{remark}

\begin{proof}[of Thm.~\ref{ThmFundamentalEstimate}] We drop the dependence on $\Mod$ in the notation. First consider the cochain $\Phi_T$. The operator $\Phi_T(\theta_1, \dots, \theta_N)$ is a sum of operators $A_0\{A_1, \dots, A_N\}_{TQ^2}$, with $A_0 = 1$ and each $A_j$, $j\geq 1$ equal to the second or third term in \eqref{ComponentsOfF}. Now by the properties \eqref{EstimatesLocallyConvexAlgebra} of the locally convex topology and the assumptions (A2) on the Fredholm module, there exists a continuous seminorm $\nu$ on $\Omega$ such that
\begin{equation} \label{NormEstimatesF}
  \|F(\theta)\Delta^{-1/2}\| \leq \nu(\theta) ~~~~\text{and}~~~~ \|F(\theta_1, \theta_2)\| \leq \nu(\theta_1)\nu(\theta_2).
\end{equation}
Hence Lemma~\ref{LemmaMainEstimate} implies that $\Phi_T(\theta_1, \dots, \theta_N)$ is a well-defined trace-class operator for each $N$. The same is true for $Q\Phi_T(\theta_1, \dots, \theta_N)$, only that in this case, $A_0 = Q$. For $\Phi_T(\theta_1, \dots, \theta_N)Q$, we can argue by passing to the adjoint.

It therefore remains to show the estimate \eqref{TraceNormEstimatePhi}. Here we consider the cochain $\Phi_T$, the proof for $Q \Phi_T$ and $\Phi_T Q$ is similar. Now notice that since $F$ has only components of arity less than two, for any partition $I \in \mathscr{P}_{M, N}$, we have 
\begin{equation*}
\{F(\theta_{I_1}), \dots, F(\theta_{I_M})\}_{TQ^2} = 0
\end{equation*}
 as soon as $I$ contains a subset $I_j$ with $|I_j| \geq 3$. For all other partitions, we have $|I_j|=1$ for $2M-N$ indices $j$ and $|I_j|=2$ for $N-M$ indices $j$. Hence by \eqref{NormEstimatesF}, we can apply Lemma~\ref{LemmaMainEstimate} with $a_j=1/2$ for $2M-N$ and $a_j=0$ for $N-M$ indices $j$. This gives
\begin{equation*}
  \bigl\|\{F(\theta_{I_1}), \dots, F(\theta_{I_M})\}_{TQ^2}\bigr\|_1 \leq \frac{e^{T/2} \Tr(e^{-TQ^2/2})}{\Gamma\left(\tfrac{3}{2}N+1-M\right)} \left(\frac{\sqrt{\pi}}{eT}\right)^{(2M-N)/2}\nu(\theta_1)\cdots \nu(\theta_N),
\end{equation*}
as $\Gamma(1/2) = \sqrt{\pi}$. With a view on \eqref{ExplicitPhiT}, we therefore get
\begin{equation*}
\begin{aligned}
  \bigl\|\Phi_T(\theta_1, \dots, \theta_N)\bigr\|_1 &\leq \sum_{M = \lceil N/2 \rceil}^N T^M \sum_{I \in \mathscr{P}_{M, N}} \frac{e^{T/2} \Tr(e^{-TQ^2/2})}{\Gamma\left(\tfrac{3}{2}N+1-M\right)}  \left(\frac{\sqrt{\pi}}{eT}\right)^{(2M-N)/2} \nu(\theta_1)\cdots \nu(\theta_N)\\
  &\leq T^{N/2}\frac{e^{T/2} \Tr(e^{-TQ^2/2})}{\Gamma(N/2)} \sum_{M=\lceil N/2 \rceil}^N \sum_{I \in \mathscr{P}_{M, N}}  \left(\frac{\sqrt{\pi}}{e}\right)^{(2M-N)/2} \nu(\theta_1)\cdots \nu(\theta_N).
  \end{aligned}
\end{equation*}
The sum over $M$ and $I$ grows at most like $C^N$ with $N$, where $C>0$ is some constant, hence it can be absorbed into the seminorm (by replacing it by a larger one). 
\end{proof}

\begin{proof}[of Thm.~\ref{ThmBianchiPhi}] Again we drop the dependence on $\Mod$ in the notation. For the proof, we introduce for $\varepsilon \geq 0$ the Bracket $\{A_1, \dots, A_N\}_{Q^2}^\varepsilon$, which is defined by the same formula \eqref{FormulaBracket}, integrating over
\begin{equation*}
  \Delta_N^\varepsilon = \bigl\{ \tau \in \Delta_N \mid \tau_j - \tau_{j-1} \geq \varepsilon~~ \text{for all}~j\bigr\}
\end{equation*}
but instead of $\Delta_N$, the $N$-simplex with the $\varepsilon$-neighborhood of the diagonals removed. Due to the fact that for any $m$, the norm of the operator $Q^k e^{-tQ^2}$ is bounded uniformly on $[\varepsilon, 1]$, the expression $\{A_1, \dots, A_N\}_{Q^2}^\varepsilon$ gives a well-defined operator for any collection of operators $A_j$ of the form $A_j = Q^n A_j^\prime Q^m$, with $A_j^\prime$ bounded. We denote by $\Phi_T^\varepsilon$ the cochain defined by the same formula \eqref{PerturbationSeriesPhiT} as $\Phi_T$, but using this bracket instead of the previous one. Now for any $\varepsilon \geq 0$, we have
\begin{equation} \label{DeltaPhiTEpsilon}
  \delta \Phi_T^\varepsilon = \sum_{N=0}^\infty (-T)^N \sum_{k=1}^N\bigl\{\underbrace{{F}^{\geq 1}, \dots, {F}^{\geq 1}}_{N-k}, \delta F^{\geq 1}, \underbrace{F^{\geq 1}, \dots, F^{\geq 1}}_{k-1} \bigr\}^\varepsilon_{TQ^2}.
\end{equation}
By the Bianchi identity \eqref{BianchiIdentity} for $F$, we have
\begin{equation} \label{FollowsFromBI}
  \delta F^{\geq 1} = \delta F = - [\omega, F] = [Q, F^{\geq 1}] - [\cc, F^{\geq 1}] - [\cc, Q^2] .
\end{equation}
At this point, this identity is just formal as both side are cochains with values in unbounded operators and we do not have enough control over their domains; however, it is easy to verify that
\begin{equation} \label{LongerCalculation}
\begin{aligned}
  \delta \Phi_T^\varepsilon &= \sum_{N=1}^\infty (-T)^N \sum_{k=1}^N\bigl\{\underbrace{{F}^{\geq 1}, \dots, {F}^{\geq 1}}_{k-1}, \delta F^{\geq 1}, \underbrace{F^{\geq 1}, \dots, F^{\geq 1}}_{N-k} \bigr\}^\varepsilon_{TQ^2}\\
  &= \sum_{N=1}^\infty (-T)^N \Bigg(\sum_{k=1}^N\bigl\{\underbrace{{F}^{\geq 1}, \dots, {F}^{\geq 1}}_{k-1}, [Q, F^{\geq 1}], \underbrace{F^{\geq 1}, \dots, F^{\geq 1}}_{N-k} \bigr\}^\varepsilon_{TQ^2}\\
  &~~~~~~~~~~~~~~~~~-\sum_{k=1}^N\bigl\{\underbrace{{F}^{\geq 1}, \dots, {F}^{\geq 1}}_{k-1}, [\cc, F^{\geq 1}], \underbrace{F^{\geq 1}, \dots, F^{\geq 1}}_{N-k} \bigr\}^\varepsilon_{TQ^2}\\
  &~~~~~~~~~~~~~~~~~-\sum_{k=1}^N\bigl\{\underbrace{{F}^{\geq 1}, \dots, {F}^{\geq 1}}_{k-1}, [\cc, Q^2], \underbrace{F^{\geq 1}, \dots, F^{\geq 1}}_{N-k} \bigr\}^\varepsilon_{TQ^2}\Bigg),
\end{aligned}
\end{equation}
formally obtained by plugging the right hand of \eqref{FollowsFromBI} side into \eqref{DeltaPhiTEpsilon}, is in fact an identity of well-defined cochains provided that $\varepsilon >0$, by the discussion of the $\varepsilon$-bracket above. Notice that since taking the commutator with $Q$ is a derivation, we have
\begin{equation*}
\bigl\{\underbrace{{F}^{\geq 1}, \dots, {F}^{\geq 1}}_{k-1}, [Q, F^{\geq 1}], \underbrace{F^{\geq 1}, \dots, F^{\geq 1}}_{N-k} \bigr\}^\varepsilon_{TQ^2} = \Bigl[Q, \bigl\{\underbrace{{F}^{\geq 1}, \dots, {F}^{\geq 1}}_{N}\}^\varepsilon_{TQ^2}\Bigr],
\end{equation*}
which in fact makes sense even for $\varepsilon = 0$, by Lemma~\ref{ThmFundamentalEstimate} above. Also the second term on the right hand side of \eqref{LongerCalculation} makes sense for $\varepsilon = 0$; for the third term, we use the following lemma.

\begin{lemma} \label{Lemma:IntByParts}
Let $A_1, \dots, A_N$ be a collection of operators of the form $A_j = Q^m A_j^\prime Q^n$ with $A_j^\prime$ bounded and assume that $A_k = A_k^\prime$ for some $k$. Then for $\varepsilon >0$, we have the identity
 \begin{equation*}
 \begin{aligned}
  &T\bigl\{A_1, \dots, [Q^2, A_k], \dots, A_N\bigr\}_{TQ^2}^\varepsilon \\
  &~~~~~~~~~~= \bigl\{A_1, \dots, A_{k}e^{-\varepsilon Q^2} A_{k+1}, \dots, A_N\bigr\}^\varepsilon_{ TQ^2} - \bigl\{A_1, \dots, A_{k-1}e^{-\varepsilon Q^2}A_{k}, \dots, A_N\bigr\}^\varepsilon_{TQ^2}.
 \end{aligned}
 \end{equation*}
Here, if $k=1$, the second term on the right-hand-side is $A_1\{A_{2}, \dots, A_N\}^\varepsilon_{TQ^2}$ instead, while if $k=N$, the first term is $\{A_1, \dots, A_{N-1}\}^\varepsilon_TA_N$.
\end{lemma}

\begin{proof}
  This follows from integration by parts, after realizing that the term on the left hand side equals
  \begin{equation*}
  \int_{\Delta_{N}^\varepsilon} \frac{\partial}{\partial {\tau_k}} \left[e^{-T(1-\tau_N)Q^2} \prod_{j=1}^N A_j e^{-T(\tau_j-\tau_{j-1})Q^2}\right] \dd \tau.
  \end{equation*}
\end{proof}

Hence for any $\varepsilon >0$, we have
\begin{equation*}
\begin{aligned}
\bigl\{\underbrace{{F}^{\geq 1}, \dots, {F}^{\geq 1}}_{k-1},[\cc, Q^2], \underbrace{F^{\geq 1}, \dots, F^{\geq 1}}_{N-k} \bigr\}_{TQ^2} &= \bigl\{\underbrace{{F}^{\geq 1}, \dots, {F}^{\geq 1}}_{k-1},\cc e^{-\varepsilon Q^2}F^{\geq 1}, \underbrace{F^{\geq 1}, \dots, F^{\geq 1}}_{N-k-1} \bigr\}_{TQ^2} \\
&- \bigl\{\underbrace{{F}^{\geq 1}, \dots, {F}^{\geq 1}}_{k-2},F^{\geq 1} e^{-\varepsilon Q^2}\cc, \underbrace{F^{\geq 1}, \dots, F^{\geq 1}}_{N-k} \bigr\}_{TQ^2}.
\end{aligned}
\end{equation*}
Now notice that the right hand side of the above identity makes sense for $\varepsilon = 0$ as well. Hence after replacing the last term in \eqref{LongerCalculation} with this term, we may take the limit $\varepsilon \rightarrow 0$ in the identity \eqref{LongerCalculation}. This gives
\begin{equation*}
\begin{aligned}
  \delta \Phi_T = [Q, \Phi_T] &- \sum_{N=1}^\infty (-T)^N \sum_{k=1}^N\bigl\{\underbrace{{F}^{\geq 1}, \dots, {F}^{\geq 1}}_{k-1}, [\cc, F^{\geq 1}], \underbrace{F^{\geq 1}, \dots, F^{\geq 1}}_{N-k} \bigr\}_{TQ^2}\\
 &- \sum_{N=1}^\infty (-T)^{N-1} \Biggl(  \cc\{\underbrace{F^{\geq 1}, \dots, F^{\geq 1}}_{N-1}\}_{TQ^2} - \{\underbrace{{F}^{\geq 1}, \dots, {F}^{\geq 1}}_{N-1}\}_{TQ^2}\cc \\
 &~~~~~~~~~~~~~~~~~~~~~~+\sum_{k=2}^{N} \bigl\{\underbrace{{F}^{\geq 1}, \dots, {F}^{\geq 1}}_{k-2}, F^{\geq 1} \cc, \underbrace{F^{\geq 1}, \dots, F^{\geq 1}}_{N-k} \bigr\}_{TQ^2} \\
 &~~~~~~~~~~~~~~~~~~~~~~- \sum_{k=1}^{N-1}\bigl\{\underbrace{{F}^{\geq 1}, \dots, {F}^{\geq 1}}_{k-1}, \cc F^{\geq 1}, \underbrace{F^{\geq 1}, \dots, F^{\geq 1}}_{N-k-1} \bigr\}_{TQ^2}\Biggr).
\end{aligned}
\end{equation*}
Finally, after an index shift, we obtain that the third term on the right hand side equals
\begin{equation*}
\begin{aligned}
&-\sum_{N=1}^\infty (-T)^{N-1} \Bigg( \cc\{\underbrace{F^{\geq 1}, \dots, F^{\geq 1}}_{N-1}\}_{TQ^2} - \{\underbrace{{F}^{\geq 1}, \dots, {F}^{\geq 1}}_{N-1}\}_{TQ^2}\cc \\
&~~~~~~~~~~~~~~~~~~~~~~~~~~~+\sum_{k=1}^{N-1}\bigl\{\underbrace{{F}^{\geq 1}, \dots, {F}^{\geq 1}}_{k-1}, {F}^{k\geq 1}\cc - \cc F^{\geq 1}, \underbrace{F^{\geq 1}, \dots, F^{\geq 1}}_{N-k-1} \bigr\}_{TQ^2} \Bigg)
\end{aligned}
\end{equation*}
\begin{equation*}
\begin{aligned}
&= - [\cc, \Phi_T] + \sum_{N=2}^\infty (-T)^{N-1} \sum_{k=1}^{N-1}\bigl\{\underbrace{{F}^{\geq 1}, \dots, {F}^{\geq 1}}_{k-1}, [\cc, F^{\geq 1}], \underbrace{F^{\geq 1}, \dots, F^{\geq 1}}_{N-k-1} \bigr\}_{TQ^2}
\end{aligned}
\end{equation*}
Combining this with the identity before, we obtain
\begin{equation*}
  \delta \Phi_T = [Q, \Phi_T] - [\cc, \Phi_T] = - [\omega, \Phi_T],
\end{equation*}
which was the claim.
\end{proof}

\section{The Chern Character }\label{SectionChernCharacter}

Throughout this section, let $\Omega$ be a locally convex dg algebra.

\begin{definition}[Chern character]
For any  $\vartheta$-summable weak Fredholm module $\Mod$, we define an even cochain $\Ch_{\Mod}\in \mathrm{Hom}^+(\mathsf{C}(\Omega), \C)$ by
\begin{equation} \label{DefinitionChernCharacterMod}
   \Ch_\Mod (\theta_0, \dots, \theta_N) =  \Str\Bigl( \cc(\theta_0) \Phi_1^\Mod(\theta_1, \dots, \theta_N)\Bigr)
\end{equation}
for $\theta_0, \dots, \theta_N \in {\Omega}$. Here $\Phi_T^\Mod$ is the quantization map defined in \S\ref{SectionQuantizationMap} and $\mathrm{Str}$ denotes the supertrace for trace-class operators on the $\Z_2$-graded Hilbert space $\Hil$. The cochain $\Ch_\Mod$ is called the \emph{Chern character} of $\Mod$. 
\end{definition}

Note that $\Ch_\Mod$ is well-defined: First, that the operator $\cc(\theta_0) \Phi_1^\Mod(\theta_1, \dots, \theta_N)$ is trace-class follows from the fundamental estimate from Thm.~\ref{ThmFundamentalEstimate}. Secondly, $\Ch_\Mod$ is a well-defined element of $\mathrm{Hom}(\CC(\Omega), \C)$. To see this, one has to observe that
\begin{equation*}
\Phi_1^\Mod(\theta_1, \dots, \theta_{j}, \mathbf{1}, \theta_{j+1}, \dots, \theta_N) = 0
\end{equation*}
for all $\theta_1, \dots, \theta_N \in \Omega$ and all $0\leq j\leq N$. However, this follows from the assumption that $\cc(\mathbf{1}) = 1$, with a view on the explicit formula \eqref{ComponentsOfF} for $F_\Mod$. Finally, $\Ch_\Mod$ is even, that is, it vanishes on odd elements of $\mathsf{C}(\Omega)$. This follows directly from the fact that both $\cc$ and $\Phi_1^\Mod$ are parity-preserving (c.f.~\ref{RemarkPhiEven}), since the supertrace vanishes on odd operators. The following was stated as Thm.~B in the introduction.

\begin{theorem}[Analyticity] \label{ThmEntire} 
The Chern character $\Ch_\Mod$ is an analytic cochain.
\end{theorem}

\begin{proof} 
That $\Ch_\Mod$ satisfies the necessary estimates \eqref{GrowthConditionAnalytic} follows directly from the fundamental estimate, Thm.~\ref{ThmFundamentalEstimate}.
\end{proof}

We now proceed to show the closedness of $\Ch_\Mod$, stated as Thm.~A above.

\begin{theorem}[Closedness] \label{ThmChClosed} The Chern character $\Ch_\Mod$ is closed, that is,
\begin{equation*}
(\underline{d}+\underline{b}+\underline{B})\Ch_\Mod = 0.
\end{equation*}
\end{theorem}

A conceptual way to prove Thm.~\ref{ThmChClosed} is to exploit the coalgebra structure on the space $\Hom(\B(\Omega), \Lin(\Hil))$ of $\Lin(\Hil)$-valued bar cochains, together with derived constructions. The result then follows from calculations similar to those of Quillen \cite[Section~8]{Quillen}. However, instead of introducing the algebraic machinery to adapt Quillen's arguments, we use the trick of passing to the acyclic extension $\Omega_{\T}$ of $\Omega$. 

Explicitly, let $\Mod_{\T}:=(\Hil, \cc_{\T}, Q)$ be the extended $\vartheta$-summable Fredholm module over the dg algebra $\Omega_{\T} := \Omega[\sigma]$ from Example~\ref{ExampleExtension}.
%Remember that $\Omega_{\T}$ carries the $d_{\T}$ (given by formula \eqref{DifferentialOmegaT}), where $\sigma$ is a formal variable of degree $-1$ with $\sigma^2=0$. 
Observe that it suffices to prove Thm.~\ref{ThmChClosed} for $\Mod_{\T}$ over $\Omega_\T$: Indeed, the inclusion $\Omega \subset  \Omega_{\T}$ induces a map
\begin{equation*}
j: \mathsf{C}(\Omega) \longrightarrow \mathsf{C}(\Omega_{\T}), 
\end{equation*}
that commutes with all differentials, and one has $\Ch_\Mod = j^*\Ch_{\Mod_{\T}}$. Hence the result for $\Ch_{\Mod_\T}$ implies that for $\Ch_\Mod$. 

We note that $\Phi_1^{\Mod_\T}$ and $F_{\Mod_{\T}}$ descend to the quotient\footnote{Here of course, $\B(\underline{\Omega}_{\mathbb{T}})$ denotes the space given by formula \eqref{FormulaBOmega}, but with the {\em space} $\underline{\Omega}_\T$ instead of the {\em algebra} $\Omega_\T$.The differentials $d$ and $b^\prime$ descend to this quotient.} $\B(\underline{\Omega}_\T)$; 
let us denote by $\underline{\Phi}^{\Mod_\T}_1$ the quotient map induced by $\Phi^{\Mod_\T}_1$. 
For the acyclic extension, the nonzero components of the curvature are given by
\begin{equation} \label{ComponentsOfExtendedF}
\begin{aligned}
  &F^{(0)}_{\Mod_\T} & &\!\!\!\!\!= Q^2,\\
  &F_{\Mod_\T}^{(1)}(\theta)& &\!\!\!\!\!= \cc(d \theta^\prime) - [Q, \cc(\theta^\prime)] - \cc(\theta^{\prime\prime})\\
  &F_{\Mod_\T}^{(2)}(\theta_1, \theta_2)& &\!\!\!\!\!= (-1)^{|\theta_1^\prime|}\bigl(\cc({\theta}_1^\prime{\theta}_2^\prime) - \cc({\theta}_1^\prime)\cc({\theta}_2^\prime)\bigr),
\end{aligned}
\end{equation}
where again the elements of $\theta \in \Omega_{\T}$ have been written as $\theta = \theta^\prime + \sigma \theta^{\prime\prime}$. 

The reason we pass to the acyclic extension is that this allows to relate the cyclic complex and the bar complex using the parity-preserving map
\begin{equation} \label{DefinitionAlpha}
  \alpha: \CC(\Omega_{\mathbb{T}}) \longrightarrow \B(\underline{\Omega}_{\mathbb{T}}), ~~~~(\theta_0, \dots, \theta_N) \longmapsto \underline{\mathbf{N}}(\sigma\theta_0, \theta_1, \dots, \theta_N),
\end{equation}
where 
\begin{equation*}
\underline{\mathbf{N}}:  \B(\underline{\Omega}_{\mathbb{T}})\longrightarrow \B(\underline{\Omega}_{\mathbb{T}})
\end{equation*}
 is the quotient map of the {\em averaging operator}, given by
\begin{equation*}
   \mathbf{N}: \B(\Omega_{\mathbb{T}})\longrightarrow \B(\Omega_{\mathbb{T}}), ~~~~(\theta_1, \dots, \theta_N) \longmapsto \sum_{k=1}^N (-1)^{n_k(n_N-n_k)} (\theta_{k+1}, \dots, \theta_N, \theta_1, \dots, \theta_k),
\end{equation*}
with $n_k = |\theta_1| + \dots + |\theta_k|-k$. 

\begin{proposition} \label{PropositionPullbackOfChern}
We have
\begin{equation} \label{PullbackOfChern}
 \Ch_{\Mod_\T} = - \alpha^*\Str(\underline{\Phi}^{\Mod_\T}_1\bigr).
\end{equation}
\end{proposition}

The proof of this proposition is based on the following lemma.

\begin{lemma} \label{LemmaCyclicPermutation}
For operators $A_0, \dots, A_N$ satisfying the assumptions of Lemma~\ref{LemmaMainEstimate}, we have
\begin{equation*}
  \sum_{j=0}^N (-1)^{k_j} \Str\bigl(\{A_{j+1}, \dots, A_N, A_0, \dots, A_j\}_{Q^2}\bigr) = \Str\bigl(A_0\{A_1, \dots, A_N\}_{Q^2} \bigr),
\end{equation*}
where $k_j = (|A_0|+\dots+|A_j|)(|A_{j+1}|+\dots+|A_N|)$. 
\end{lemma}

\begin{proof}
First notice that by the cyclic permutation property of the trace, we have after a substitution in $\Delta_N$ that
\begin{equation*}
\Str\bigl(A_0\{A_1, \dots, A_N\}_{Q^2}\bigr) = (-1)^{k_j} \Str\bigl(A_{j+1}\{A_{j+2}, \dots, A_N, A_0, \dots, A_j\}_{Q^2}\bigr)
\end{equation*}
for any $j=0, \dots, N$. A special case of this is
\begin{equation*} 
  \Str\bigl(A_0\{A_1, \dots, A_j, \id, A_{j+1}, \dots, A_N\}_{Q^2}\bigr) = (-1)^{k_j} \Str\bigl(\{A_{j+1}, \dots, A_N, A_0, \dots, A_j\}_{Q^2}\bigr).
\end{equation*}
On the other hand, the semigroup property and integration by parts yields
\begin{equation*}
\begin{aligned}
  \{A_1&, \dots, A_j, \id, A_{j+1}, \dots, A_N\}_{Q^2} = \int_{\Delta_{N}} (\tau_{j+1} -\tau_j) e^{-\tau_1Q^2} \prod_{j=1}^N A_j e^{-(\tau_{j+1}-\tau_{j})Q^2} \dd \tau
\end{aligned}
\end{equation*}
Therefore, using that $\sum_{j=0}^N (\tau_{j+1} - \tau_j) = 1$, we obtain
\begin{equation} \label{LemmaSpecialCaseId}
\sum_{j=0}^N \{A_1, \dots, A_j, \id, A_{j+1}, \dots, A_N\}_{Q^2} = \{A_1, \dots, A_N\}_{Q^2}.
\end{equation}
Combining this with the observation before yields the claim.
\end{proof}

\begin{proof}[of Prop.~\ref{PropositionPullbackOfChern}]
By definition of $\alpha$, we have
\begin{equation*}
  \alpha^*\Str\bigl(\underline{\Phi}_1^{\Mod_{\T}}(\theta_1, \dots, \theta_N)\bigr) 
  = \sum_{k=0}^{N+1} (-1)^{m_{k-1}(m_N-m_{k-1})}\Str\bigl(\Phi_1^{\Mod_{\T}}(\theta_{k}, \dots, \theta_N, \sigma \theta_0, \theta_1, \dots, \theta_{k-1})\bigr),
\end{equation*}
where $m_k = |\theta_0| + \dots + |\theta_k| - k$. Observe that by \eqref{ComponentsOfExtendedF} we have 
\begin{equation} \label{ComponentsFWithSigma}
F_{\Mod_{\T}}(\theta_N, \sigma\theta_0) = F_{\Mod_{\T}}(\sigma \theta_0, \theta_1) = 0 \qquad \text{and} \qquad F_{\Mod_{\T}}(\sigma\theta_0) = - \cc(\theta_0^\prime).
\end{equation} 
Therefore, when writing $\Phi_1^{\Mod_{\T}}$ as a sum over partitions, the right hand side of the first equation in this proof can be written as
\begin{equation*}
- \sum_{k=0}^{N+1} (-1)^{m_{k-1}(m_N-m_{k-1})} \sum_{M=1}^N \sum_{\ell=1}^{M+1} \sum_{I \in \mathscr{P}_{M, N}^{k, \ell}} 
\begin{aligned}\Str&\bigl\{F_{\Mod_{\T}}(\theta_{I_{\ell}}), \dots \\
& \dots, F_{\Mod_{\T}}(\theta_{I_M}), \cc(\theta_0^\prime), F_{\Mod_{\T}}(\theta_{I_1}), \dots \\
&~~~~~~~~~~~~~~~~~~~~~\dots, F_{\Mod_{\T}}(\theta_{I_{\ell-1}}) \bigr\}_{Q^2}, \end{aligned}
\end{equation*}
where $\mathscr{P}^{k, \ell}_{M, N} \subset \mathscr{P}_{M, N}$ is the set of partitions such that for some $I_1, \dots, I_{\ell-1} \subset \{1, \dots, k-1\}$ and $I_{\ell}, \dots, I_M \subset \{k, \dots, M\}$. 
It follows from \eqref{ComponentsOfExtendedF} that for each $I \in \mathscr{P}_{M, N}^{k, \ell}$, we have 
\begin{equation*}
  m_k = |\theta_0^\prime| + |F_{\Mod_{\T}}(\theta_{I_1})| + \dots + |F_{\Mod_{\T}}(\theta_{I_\ell})| =: \tilde{m}_\ell.
\end{equation*}
Moreover, for each choice of $\ell$ and $M \leq N$, the sets $\mathscr{P}_{M, N}^{k, \ell}$ are pairwise disjoint (for different choices of $k$), and their union (over $k$) is $\mathscr{P}_{M, N}$.
Hence the sums over $k$ and $\mathscr{P}_{M, N}^{k, \ell}$ above combine to one sume over $\mathscr{P}_{M, N}$, giving
\begin{equation*}
- \sum_{M=1}^N \sum_{I \in \mathscr{P}_{M, N}} \sum_{\ell=1}^M (-1)^{\tilde{m}_\ell(m_N - \tilde{m}_\ell)} \begin{aligned}\Str&\bigl\{F_{\Mod_{\T}}(\theta_{I_{\ell}}), \dots \\
& \dots, F_{\Mod_{\T}}(\theta_{I_M}), \cc(\theta_0^\prime), F_{\Mod_{\T}}(\theta_{I_1}), \dots \\
&~~~~~~~~~~~~~~~~~~~~~\dots, F_{\Mod_{\T}}(\theta_{I_{\ell-1}}) \bigr\}_{Q^2}.\end{aligned}
\end{equation*}
Finally, by Lemma~\ref{LemmaCyclicPermutation}, this equals
\begin{equation*}
- \Str\bigl(\cc(\theta_0^\prime), \Phi_1^{\Mod_{\T}}(\theta_1,  \dots, \theta_N)\bigr) = - \Ch_\Mod(\theta_0, \dots, \theta_N).
\end{equation*}
This finishes the proof.
\end{proof}

\begin{proof}[of Thm.~\ref{ThmChClosed}] 
Because of Prop.~\ref{PropositionPullbackOfChern},
 we are interested in the interaction of the map $\alpha$ from \eqref{DefinitionAlpha} with respect to the various differentials. Here straightforward calculations show that
\begin{align}\label{interaction}
  \underline{b}^\prime \alpha + \alpha \underline{b} = 0, ~~~~ \underline{d}_{\mathbb{T}} \alpha + \alpha \underline{d}_{\mathbb{T}} = -h, ~~~~ \alpha \underline{B} = \underline{S}h,
\end{align}
where $\underline{b}^\prime$, $\underline{d}_{\mathbb{T}}$ are the differentials on the quotient $\B(\underline{\Omega}_{\mathbb{T}})$, $h$ is given by
\begin{equation*}
  h: \CC(\Omega_{\mathbb{T}}) \longrightarrow \B(\underline{\Omega}_{\mathbb{T}}), ~~~~(\theta_0, \dots, \theta_N) \longmapsto \underline{\mathbf{N}}(\theta_0, \theta_1, \dots, \theta_N),
\end{equation*}
and 
\begin{equation*}
\underline{S}:    \B(\underline{\Omega}_{\mathbb{T}})\longrightarrow \B(\underline{\Omega}_{\mathbb{T}})
\end{equation*}
 is the quotient map of the map $S$ considered in \eqref{DefinitionOfOperatorS}. We obtain that
\begin{equation} \label{CompatibilityAlpha}
  \alpha(\underline{d}_{\mathbb{T}} + \underline{b} + \underline{B}) = (\underline{d}_{\mathbb{T}} + \underline{b}^\prime)\alpha + (\underline{S}-1)h,
\end{equation}
so with a view on \eqref{PullbackOfChern}, we have that
\begin{equation} \label{FormulaChPhi}
  -(\underline{d}_{\mathbb{T}} + \underline{b} + \underline{B})\Ch_{\Mod_\T} = \Str \bigl(\Phi_1^{\Mod_\T}(\underline{d}_{\mathbb{T}} + \underline{b}^\prime)\alpha \bigr) + \Str(\Phi_1^{\Mod_\T}(\underline{S}-1)h)
\end{equation}
Denote by $\B^\natural(\Omega_{\mathbb{T}}) \subset \B(\Omega_{\mathbb{T}})$ the image of $\mathbf{N}$, which is the space of cyclic bar chains. On this subspace, $\Str (\Phi^{\M_\T}_1)$ is coclosed by Thm.~\ref{ThmBianchiPhi}, meaning that
\begin{equation} \label{StrPhiClosed}
  \Str \Phi^{\Mod_\T}_1\bigl((d_\T+b^\prime)c\bigr) = 0 
\end{equation}
for all bar chains $c \in \B^\natural(\Omega_{\T})$. This follows from the fact that $\Phi_1^{\Mod_{\T}}((d_\T+b^\prime)\theta)$ is a sum of super-commutators of operators on $\Hil$ in this case, as can be seen from the explicit formula \eqref{ExplicitThmBianchiPhi}.

By \eqref{StrPhiClosed}, the first summand on the right hand side of \eqref{CompatibilityAlpha} is zero, since $\alpha$ takes values in $\B^\natural(\underline{\Omega}_{\mathbb{T}})$. On the other hand, since $F(\sigma) = -1$, it follows from the special case \eqref{LemmaSpecialCaseId} of Lemma~\ref{LemmaCyclicPermutation} that on the subspace $\B^\natural(\Omega_\T)$, we have
\begin{equation} \label{SminusId}
\begin{aligned}
   \Str(\Phi^{\Mod_\T}_1S) &= - \sum_{N=0}^\infty (-1)^{N+1} \sum_{k=0}^N\Str\bigl\{\underbrace{F_{\Mod_\T}^{\geq 1}, \dots, F_{\Mod_\T}^{\geq 1}}_{k}, 1, \underbrace{F_{\Mod_\T}^{\geq 1}, \dots, F_{\Mod_\T}^{\geq 1}}_{N-k}\bigr\}_{Q^2}\\
   &= - \sum_{N=0}^\infty (-1)^{N+1} \Str\bigl\{\underbrace{F_{\Mod_\T}^{\geq 1}, \dots, F_{\Mod_\T}^{\geq 1}}_{N}\bigr\}_{Q^2} =  \Str(\Phi^{\Mod_\T}_1).
\end{aligned}
\end{equation}
Since also $h$ takes values in $\B^\natural(\underline{\Omega}_\T)$, this shows that the second summand on the right hand side of \eqref{FormulaChPhi} vanishes. This finishes the proof.
\end{proof}

We close this section by showing that the Chern character descends to the Chen normalized complxes in the case that $\Mod$ is a $\vartheta$-summable Fredholm module, not only a {\em weak} Fredholm module, i.e., when $\Mod$ satisfies the identities \eqref{Multiplicativity}.

\begin{theorem}[Chen normalization]\label{ThmNormalized} 
The Chern character $\Ch_{\Mod_\T}$ of a $\vartheta$-summable Fredholm module is Chen normalized. 
\end{theorem}

\begin{proof}
We have to show that $\Ch_{\Mod_\T}$ vanishes on $\DD^\T(\Omega)$. To begin with, the fact that $\Ch_{\Mod_\T}$ vanishes on the image of $S-\mathrm{id}$ and $R$ is true also for weak Fredholm modules, as it follows from the calculation \eqref{SminusId} above, respectively directly from the definition.

It remains to show that if \eqref{Multiplicativity} holds, then for all $f \in \Omega^0$ and each $i \geq 1$, we have $\Ch_{\Mod_\T}S_i^{(f)}= 0$ as well as $\Ch_{\Mod_\T} [\underline{d}+\underline{b}, S_i^{(f)}] = 0$. The first identity follows directly from the definition of $\Phi_T^{\Mod_\T}$, since the assumptions \eqref{Multiplicativity} imply that 
  \begin{equation*}
  F(f) = 0 ~~~~\text{and}~~~~  F(f, \theta) = F(\theta, f) = 0
  \end{equation*}
  for any $\theta \in \Omega$. To see the second identity, observe first that
  \begin{equation} \label{IdentityInD}
  \begin{aligned}
  [(\underline{d}+\underline{b}), S_i^{(f)}](\theta_0, \dots, \theta_{N}) = &(\theta_0, \dots, \theta_{i}, df, \theta_{i+1}, \dots, \theta_N) \\
  &- (\theta_0, \dots, \theta_{i}f, \theta_{i+1}, \dots, \theta_N) \\
  &+ (\theta_0, \dots, \theta_{i}, f\theta_{i+1}, \dots, \theta_N)
  \end{aligned}
  \end{equation}
  for $0 \leq i \leq N-1$, while for $i = N$, we have
\begin{equation} \label{IdentityInD2}
  \begin{aligned}
  [(\underline{d}+\underline{b}), S_N^{(f)}](\theta_0, \dots, \theta_{N}) = &(\theta_0, \dots, \theta_N, df) - (\theta_0, \dots, \theta_{N}f) + (f\theta_0, \dots, \theta_{N}).
  \end{aligned}
  \end{equation}
  To compute $\Phi_1^{\Mod_\T}$ of this, we use the identities
  \begin{equation*}
  \begin{aligned}
    F(df) &= - [Q^2, \cc(f)],\\
    F(df, \theta) + F(f\theta) &= \cc(f)F(\theta),\\
    F(\theta, df) - F(\theta f) &=  F(\theta)\cc(f),
  \end{aligned}
  ~~~~~~~
   \begin{aligned}
     F(f\theta_1, \theta_2) &= \cc(f)F(\theta_1, \theta_2),\\
     F(\theta_1f, \theta_2) &= \cc(f)F(\theta_1, \theta_2),\\
     F(\theta_1 f, \theta_2) &= F(\theta_1, f\theta_2),
  \end{aligned}
  \end{equation*}
  which follow directly from \eqref{Multiplicativity}. The theorem is then a consequence of the integration by parts lemma, Lemma~\ref{Lemma:IntByParts}, after a careful investigation of the terms appearing in the formula \eqref{ExplicitPhiT}.
  \end{proof}

\section{Homotopy Invariance of the Chern Character} \label{SectionHomotopy}

In this section, we show that the Chern character defined above is invariant under suitable deformations of Fredholm modules. While the results of this section are not necessary for the remainder of the paper, we believe them to be of independent interest, in particular for the proof of the localization formula for the Chern character in the case of Example~\ref{ExampleSpinors}, see \S\ref{SectionPathIntegral} below. Throughout, let $\Omega$ be a locally convex dg algebra. 

\begin{definition}[Homotopy of Fredholm modules] \label{DefHomotopyFredholm}
~ A {\em homotopy} of $\vartheta$-summable (weak) Fredholm modules over $\Omega$ is a family $\Mod^s = (\Hil, \cc^s, Q_s)$ of $\vartheta$-summable (weak) Fredholm modules, ${s\in [0,1]}$, satisfying the following conditions.
\begin{enumerate}
\item[(H1)] The seminorms appearing in (A1), (A2) of Definition \ref{fred} can be chosen independently of $s$, and for all $T>0$, we have
\begin{equation*}
\sup_{s\in [0,1]}\mathrm{Tr}(e^{-TQ_s^2})<\infty;
\end{equation*}
\item[(H2)] For all $s \in [0,1]$, the operators $Q_s$ have the same domain of definition, and for each element $h \in \dom(Q_s)$, the map $s \mapsto Q_s h$ is a continuously differentiable curve in $\Hil$; in particular, the derivative $\dot{Q}_s$ is a densely defined operator on $\Hil$. Moreover, we require that $\dot{Q}_s\Delta_s^{-1/2}$ and $\Delta_s^{-1/2}\dot{Q}_s$ are bounded, where $\Delta_s = Q_s^2+1$, with uniform norm bound
\begin{equation*}
\sup_{s\in [0,1]}\bigl\|\Delta_s^{-1/2}\dot{Q}_s\bigr\| + \sup_{s\in [0,1]}\bigl\|\dot{Q}_s\Delta_s^{-1/2}\bigr\|<\infty;
\end{equation*}
\item[(H3)] For all $\theta\in \Omega$, the map $s\mapsto\cc^{s}(\theta)$ is continuously differentiable with respect to the strong operator topology.
\end{enumerate}
\end{definition}

The main result of this section is the following.

\begin{theorem}[Homotopy invariance] \label{TheoremHomotopyInvariance}
  If $\Mod^s = (\Hil, \cc^s, Q_s)$, ${s\in [0,1]}$, is a homotopy of $\vartheta$-summable Fredholm modules, then the Chern characters $\Ch_{\Mod^0_\T}$ and $\Ch_{\Mod^1_\T}$ are homologous in the complex $\NN^{\T, \alpha}(\Omega)$.
\end{theorem}

We start with some auxiliary constructions. To this end, we fix a homotopy  $\Mod^s = (\Hil, \cc^s, Q_s)$, ${s\in [0,1]}$, of $\vartheta$-summable (possibly weak) Fredholm modules and define a family of auxiliary $\Lin(\Hil)$-valued bar cochains
\begin{equation} \label{DefinitionPsiT}
  \Psi_T^{\Mod^s} := -T\int_0^1 {\Phi}_{uT}^{\Mod^s}\, \dot{\omega}_{\Mod^s} {\Phi}_{(1-u)T}^{\Mod^s} \,\dd u.
\end{equation}
In the following two propositions, we will abbreviate $\Psi^s_T := \Psi_T^{\Mod^s}$, $\Phi_T^s := \Phi_T^{\Mod^s}$ etc.

\begin{proposition} \label{PropPsiEntire} 
There exists a continuous seminorm $\nu$ on $\Omega$ such that for each $s \in [0, 1]$, each $T>0$, and all $\theta_1, \dots, \theta_N \in \Omega$, one has the estimate
  \begin{equation*}
     \bigl\|\Psi_T^s(\theta_1, \dots, \theta_N)\bigr\|_1 \leq e^{T/2} \mathrm{Tr}[e^{-TQ_s^2/2}]\frac{T^N}{\sqrt{N!}} \nu(\theta_1) \cdots \nu(\theta_N)
  \end{equation*}
for the trace class norms. The same is true for $\Psi_TQ_s$ and $Q_s \Psi_T$ instead of $\Psi_T$.
\end{proposition}

\begin{proof}
Again, the argument for $\Psi_T^s Q_s$ and $Q_s\Psi_T^s$ is similar to the one for $\Psi_T^s$, so we restrict ourselves to the discussion of the latter case.

For suitable operators $A_1, \dots, A_N, B$ on $\Hil$, we have the identity
\begin{equation} \label{SowingIdentity}
\begin{aligned}
  &\{A_1, \dots, A_{k}, B, A_{k+1}, \dots, A_N\}_{TQ_s^2} \\
  &~~~~~~~~~~~~~~~= \int_0^1 u^k(1-u)^{N-k}\{A_1, \dots, A_k\}_{uTQ^2}^s B \{A_{k+1}, \dots, A_N\}_{(1-u)TQ_s^2}\, \dd u.
\end{aligned}
\end{equation} 
Using this and formula \eqref{ExplicitPhiT} for the quantization map, we obtain that $\Psi_T^{s}(\theta_1, \dots, \theta_N)$ is given by the formula
\begin{equation*}
\begin{aligned}
 \sum_{M=1}^N (-T)^{M+1} &\sum_{j=1}^M \Bigg( 
   \sum_{I \in \mathscr{P}_{M, N}}\bigl\{F_s(\theta_{I_1}), \dots, F_s(\theta_{I_{j-1}}), \dot{Q}_s, F_s(\theta_{I_{j}}), \dots, F_s(\theta_{I_M})\bigr\}_{TQ_s^2}\\
  &\!\!\!+\sum_{k=1}^N\sum_{\substack{I \in \mathscr{P}_{M, N}\\I_j=\{k\}}}\bigl\{F_s(\theta_{I_1}), \dots, F_s(\theta_{I_{j-1}}), \dot{\cc}^s(\theta_k), F_s(\theta_{I_{j+1}}), \dots, F_s(\theta_{I_M})\bigr\}_{TQ_s^2}\Bigg).
\end{aligned}
\end{equation*}
 Now each of the brackets can be estimated using Lemma~\ref{LemmaMainEstimate}, where to estimate the first bracket, we also use the bound on $\dot{Q}_s\Delta_s^{-1/2}$. We can then proceed as in the proof of Thm.~\ref{ThmFundamentalEstimate}. The uniformity in $s$ follows from the assumptions on the admissible pair.
\end{proof}

\begin{proposition} \label{PropTransgression} 
We have 
\begin{equation*}
  \frac{\dd}{\dd s} \Phi_T^s = \delta \Psi_T^s  + [\omega_s, \Psi_T^{s}].
\end{equation*}
\end{proposition}

\begin{proof}
Similar to the proof of Lemma~2.2~(5) in \cite{GetzlerSzenes}, one proves the formula
\begin{equation*}
  \frac{\dd}{\dd s} \{A_1, \dots, A_N\}_{TQ_s^2} = - T \sum_{k=0}^N \bigl\{A_1, \dots, A_{k}, [Q_s, \dot{Q}_s], A_{k+1}, \dots, A_N\bigr\}_{TQ_s^2},
\end{equation*}
for operators $A_1, \dots, A_N$ on $\Hil$, provided both sides are well-defined. Using this formula and an index shift, we obtain, similar to the proof of Thm.~\ref{ThmBianchiPhi}, that
\begin{equation*}
\begin{aligned}
\frac{\dd}{\dd s} \Phi^s_T &=  \sum_{N=0}^\infty (-T)^{N+1} \sum_{k=1}^{N+1}\bigl\{\underbrace{F^{\geq 1}_s, \dots, F^{\geq 1}_s}_{k}, [Q_s, \dot{Q}_s] + \dot{F}^{\geq 1}_s, \underbrace{F^{\geq 1}_s, \dots, F^{\geq 1}_s}_{N-k} \bigr\}_{TQ_s^2}\\
&=  -T \int_0^1 \Phi^s_{uT}\dot{F}_s\Phi^s_{(1-u)T} \dd u,
\end{aligned}
\end{equation*}
where we used that $\dot{F}_s = [Q_s, \dot{Q}_s] + \dot{F_s}^{\geq 1}$, as well as the identity \eqref{SowingIdentity}. Now
\begin{equation*}
   \delta \dot{\omega}_s + [\omega_s, \dot{\omega}_s] = \frac{\dd}{\dd s} \bigl\{ \delta \omega_s + \omega^2_s \bigr\} = \dot{F}_s.
\end{equation*}
Therefore, using that $\delta \Phi^s_{T} + [\omega_s, \Phi^s_{T}]= 0$ by Thm.~\ref{ThmBianchiPhi}, we obtain
\begin{equation*}
  \int_0^1 \Phi^s_{sT}\dot{F}_s\Phi^s_{(1-s)T} \dd s =  \delta \int_0^1 {\Phi}^s_{sT}\dot{\omega}_s\Phi^s_{(1-s)T} \dd s + \Bigl[ \omega_s,  \int_0^1 {\Phi}^s_{sT}\dot{\omega}_s\Phi^2_{(1-s)T} \dd s\Bigr],
\end{equation*}
as requested.
\end{proof}

In order to prove Thm.~\ref{TheoremHomotopyInvariance}, fix a homotopy $\Mod^s := (\Hil, \cc^s, Q_s)$, ${s\in [0,1]}$, of $\vartheta$-summable Fredholm modules and let $\Mod_\T^s := (\Hil, \cc_\T^s, Q_s)$, ${s\in [0,1]}$, be the associated homotopy of $\vartheta$-summable {\em weak} Fredholm modules over $\Omega_\T$. Now using the map $\alpha$ defined in \eqref{DefinitionAlpha}, we define the corresponding {\em Chern-Simons form} by
\begin{equation} \label{DefinitionCSChA}
  \mathrm{CS}\bigl((\Mod_{\T}^s)_{s \in [0, 1]}\bigr): = - \int_0^1 \alpha^* \Str (\Psi^{\M_\T^s}_1) \dd s.
\end{equation}
By Prop~\ref{PropPsiEntire}, $\mathrm{CS}((\Mod_{\T}^s)_{s \in [0, 1]})$ is an analytic cochain, and similar to the proof of Thm.~\ref{ThmNormalized}, one shows that it is Chen normalized if $\Mod^s$ satisfies \eqref{Multiplicativity}. Thm.~\ref{TheoremHomotopyInvariance} is then a consequence of the following result.

\begin{theorem}[Transgression formula] \label{ThmTransgressionChA}
For any $T>0$, we have the transgression formula
\begin{equation*}
\begin{aligned}
 \Ch_{\Mod_\T^1} - \Ch_{\Mod_\T^0} = (\underline{d}_{\T} + \underline{b} + \underline{B})\mathrm{CS}\bigl((\Mod_{\T}^s)_{s \in [0, 1]}\bigr).
\end{aligned}
\end{equation*}
\end{theorem}

\begin{proof}
Since $-\Ch_{\Mod_\T^s} = \alpha^* \Str (\Phi^{\Mod_\T^s})  =\Str (\Phi^{\Mod_\T^s} \alpha)$ by the proof Thm.~\ref{ThmChClosed}, Prop.~\ref{PropTransgression} implies
\begin{equation*}
  - \frac{\dd}{\dd s} \Ch_{\Mod_\T^s} = \Str (\delta\Psi^{\Mod_\T^s}_1 \alpha) + \Str \left( [\omega^s, \Psi_1^{\Mod_\T^s}]\alpha \right) = \Str\left( \Psi_1^{\Mod_\T^s}(d_\T+b^\prime) \alpha\right),
\end{equation*}
where we used that $\alpha$ takes values in the space $\B^\natural(\underline{\Omega}_\T)$ of cyclic chains, on which $[\omega^s, \Psi_1^{\Mod_\T^s}]$ vanishes. Using \eqref{CompatibilityAlpha}, we then get
\begin{equation*}
 \frac{\dd}{\dd s} \Ch_{\Mod_\T^s}  = -  (\underline{d}_\T + \underline{b}+\underline{B}) \alpha^* \Str (\Psi^{\Mod_\T^s}_1) - \Str (\Psi_1^{\Mod_\T^s}(S - \mathrm{id})h).
\end{equation*}
The second summand on the right hand side vanishes by an argument similar to that in the proof Thm.~\ref{ThmChClosed}, which finishes the proof.
\end{proof}

\begin{remark}
Of course, similar results hold in the case of a homotopy $\Mod^s$, $s \in [0, 1]$, of $\vartheta$-summable {\em weak} Fredholm modules over $\Omega$. In this case $\Ch_{\Mod^0_\T}$ and $\Ch_{\Mod^1_\T}$ are homologous in the complex $\CC_\alpha(\Omega_\T)$.
\end{remark}

\section{The Bismut-Chern Character of an Idempotent} \label{SectionBChOfP}

Let $\Omega$ be a locally convex dg algebra. The algebra of $n \times n$ matrices with entries in $\Omega$ is denoted by $\mathrm{Mat}_n(\Omega)$, which is then again a locally convex dg algebra, with the differential (also denoted by $d$) acting entrywise. We can form the acyclic extension $\Omega_\T$ of $\Omega$ as in Example~\ref{ExampleExtension} by adjoining a formal variable $\sigma$ of degree $-1$. The same can be done for $\mathrm{Mat}_n(\Omega)$; we have $\mathrm{Mat}_n(\Omega)_\T \cong \mathrm{Mat}_n(\Omega_\T)$.

The goal of this section is to construct the Bismut-Chern character associated to an in idempotent $p$ in the algebra $\mathrm{Mat}_n(\Omega^0)$, which will be an even element of $\CC^\epsilon(\Omega_\T)$, closed as an element in the complex $\NN^{\T, \epsilon}(\Omega)$. To this end, let $p \in \mathrm{Mat}_n(\Omega^0)$ be an idempotent, $p^2 = p$, and write $p^\perp := \mathbf{1} - p$. The formula
\begin{equation} \label{ConnectionForP}
D\theta = p\,d(p\theta) + p^\perp d (p^\perp\theta). 
\end{equation}
then defines an operator $D$ on $\mathrm{Mat}_n(\Omega)$. 
The corresponding {\em connection form} $\varrho_p \in \mathrm{Mat}_n(\Omega^1)$ is defined by the equation $D = d+\varrho_p$. Its {\em curvature} is the element $R_p$ of $\mathrm{Mat}_n(\Omega^2)$ given by $R_p = d \varrho_p + \varrho_p^2$. We have the explicit formulas
\begin{equation*}
  \varrho_p =  p\,dp + p^\perp dp^\perp = (2p-\mathbf{1})dp, ~~~~~~ R_p = (dp)^2.
\end{equation*}
for the connection form and the curvature.

\begin{example}[Vector bundles with connection]\label{ExampleVectorBundlesAndProjections}
Given a manifold $X$, any smooth function $p: X \rightarrow \mathrm{Mat}_n(\C)$ taking values in idempotents (equivalently, an idempotent in $\mathrm{Mat}_n(C^\infty(X)) \subset \mathrm{Mat}_n(\Omega(X))$) determines a vector bundle $E := \mathrm{im}(p)$. It has a natural connection $\nabla^E$ given $\nabla^E f = p(d f)$. It is well-known \cite[Thm.~1]{NarasimhanRamanan} that any vector bundle with connection on $X$ is isomorphic (through a connection-preserving isomorphism) to one of this form.

The same construction can be applied to the complementary bundle $E^\perp =  \mathrm{im}(\mathbf{1}-p)$, and the direct sum connection $\nabla^E \oplus \nabla^{E^\perp}$ defines a connection on $\mathrm{im}(p) \oplus \mathrm{im}(\mathbf{1}-p) = \C^n$ commuting with $p$. The operator $D$ in \eqref{ConnectionForP} is the extension to differential forms of this direct sum connection.
\end{example}

Following \cite[Section~6]{GJP}, we combine these elements into a single odd element $\mathfrak{R}_p \in \mathrm{Mat}_n(\Omega_\T)$ given by
\begin{equation*}
  \mathfrak{R}_p := \varrho_p + \sigma R_p.
\end{equation*}
If we are in the situation of Example~\ref{ExampleVectorBundlesAndProjections}, where $\Omega = \Omega(X)$ and $\Omega_{\T} \cong \Omega(X \times S^1)^{S^1}$ (with $\sigma = dt$, see Remark~\ref{RemarkAcyclicAndCircle}), then $\mathfrak{R}_p \in \Omega(X \times S^1)^{S^1}$. 
The significance of this composite expression is that $d_\T + \mathfrak{R}_p$ is a flat equivariant connection on the manifold $X \times \T$; see Lemma~\ref{Lemma:IdentitiesP}.

\begin{definition}[Bismut-Chern character] \label{DefBChP}
The {\em Bismut-Chern character} of $p$ is the even entire chain $\Ch(p)\in \CC_+^\epsilon(\Omega_\T) $ given by 
\begin{equation} \label{FormulaBChP}
\Ch(p) = \sum_{N=0}^\infty (-1)^N  \tr \bigl(p , \underbrace{\mathfrak{R}_p, \dots, \mathfrak{R}_p}_{N}\bigr).
\end{equation}
 \end{definition}
 
In the definition, 
\begin{equation*}
\tr: \mathsf{C}^{\epsilon}(\mathrm{Mat}_n(\Omega_\T)) \longrightarrow \CC^{\epsilon}(\Omega_\T), \quad (\Theta_{0}, \dots, \Theta_N) \longmapsto \sum_{i_0, \dots, i_{N} = 1}^n \bigl((\Theta_{0})_{i_{0}}^{i_1}, (\Theta_{1})_{i_{2}}^{i_{1}}, \dots, (\Theta_N)_{i_{N}}^{i_{0}}\bigr)
\end{equation*}
 is the {\em generalized trace map}, which preserves all differentials, see \cite[1.2.1]{Loday}.  
 Note that $\Ch(p)$ is indeed entire: If $\nu$ is a continuous seminorm on $\Omega_\T$, then
\begin{equation*}
  \pi_{\nu, N}\bigl(p, \underbrace{\mathfrak{R}_p, \dots, \mathfrak{R}_p}_N\bigr) \leq \nu(p)\nu(\mathfrak{R}_p)^N,
  \end{equation*}
  so that clearly $\epsilon_\nu(\Ch(p)) < \infty$.
Moreover, $\Ch(p)$ is even because $\mathfrak{R}_p$ is odd in $\mathrm{Mat}_n(\Omega_\T)$, hence even in $\mathrm{Mat}_n(\Omega_\T)[1]$.

\begin{remark}
The entire chain $\Ch(p)$ was first considered by Getzler, Jones and Petrack in \cite[Section~6]{GJP}, in the Example~\ref{ExampleSpinors}, where $\Omega = \Omega(X)$, differential forms on a manifold $X$. The significance of $\Ch(p)$ is that under the extended iterated integral map (cf. \S\ref{SectionPathIntegral} below), $\Ch(p)$ is sent to the {\em Bismut-Chern character} $\mathrm{BCh}(E, \nabla)$ of the vector bundle with connection $(E, \nabla)$ corresponding to $p$ as in Example~\ref{ExampleVectorBundlesAndProjections}. This is an equivariantly closed differential form on the loop space, first considered by Bismut in \cite{Bismut1}. An odd variant of $\Ch(p)$ has been  recently constructed by the first named author and S.~Cacciatori in \cite{CacciatoriGueneysu}, producing the odd Bismut-Chern character of Wilson \cite{Wilson}.
\end{remark}

\begin{theorem} \label{ThmBChClosed} 
As an element of the Chen normalized complex $\NN^{\T, \epsilon}(\Omega)$, one has
\begin{equation*}\label{closy}
(\underline{d}_\T+\underline{b}+\underline{B})\Ch (p) = 0.
\end{equation*}
\end{theorem}

For the proof of the proof of Thm.~\ref{ThmBChClosed}, we need the following Lemma; throughout, we write $\mathfrak{R}$ instead of $\mathfrak{R}_p$.

\begin{lemma} \label{Lemma:IdentitiesP} 
We have the identities
\begin{align}
 \label{IdentityForP1}  d_\T p + [\mathfrak{R}, p] &= 0\\
  d_\T  \mathfrak{R} + \mathfrak{R}^2 &= 0\\
  d_\T(\sigma p) + [\mathfrak{R}, \sigma p] &= -p.
\end{align}
\end{lemma}

\begin{proof}
This follow from straightforward calculation, using the relations
\begin{equation*}
  [p, \varrho] = dp, \quad [\varrho, dp] = [dp, p], \quad \text{and} \quad [R, p] = 0.
\end{equation*}
The first of these latter identities follows from differentiating the equation $p^2 = p$ and straightforward calculation using the definition of $\varrho$, the second identity follows from differentiating the first, and the third is a direct calculation using the previous results.
\end{proof}

\begin{remark}
In fact, the proof shows that the element defined above is even closed {\em before} taking the trace, i.e., as an element in $\mathsf{C}^{\epsilon}(\mathrm{Mat}_n(\Omega_\T))$.
\end{remark}

\begin{proof}[of Thm.~\ref{ThmBChClosed}]
We have to show that $(\underline{d}_\T+\underline{b}+\underline{B})\Ch (p)$ is contained in the closure of $\DD^{\T}(\Omega)$ in $\CC^{\T, \epsilon}(\Omega)$. To this end, calculate
\begin{equation*}
\begin{aligned}
\underline{d}_{\mathbb{T}}(p, \underbrace{\mathfrak{R}, \dots, \mathfrak{R}}_{N}) &= \bigl({d}_\T p, \underbrace{\mathfrak{R}, \dots, \mathfrak{R}}_{N}\bigr) - \sum_{k=1}^N \bigl(p, \underbrace{\mathfrak{R}, \dots, \mathfrak{R}}_{k-1}, {d}_\T\mathfrak{R}, \underbrace{\mathfrak{R}, \dots, \mathfrak{R}}_{N-k}\bigr)  \\
\underline{b}(p, \underbrace{\mathfrak{R}, \dots, \mathfrak{R}}_{N}) &= ([p, \mathfrak{R}], \underbrace{\mathfrak{R}, \dots, \mathfrak{R}}_{N-1})+ \sum_{k=1}^{N-1} (p, \underbrace{\mathfrak{R}, \dots, \mathfrak{R}}_{k-1}, \mathfrak{R}^2, \underbrace{\mathfrak{R}, \dots, \mathfrak{R}}_{N-k-1}) .
\end{aligned}
\end{equation*}
By Lemma~\ref{Lemma:IdentitiesP}, the second  equation can be rewritten as
\begin{equation*}
\begin{aligned}
\underline{b}(p, \underbrace{\mathfrak{R}, \dots, \mathfrak{R}}_{N}) = \bigl({d}_\T p, \underbrace{\mathfrak{R}, \dots, \mathfrak{R}}_{N-1}\bigr)- \sum_{k=1}^{N-1} \bigl(p, \underbrace{\mathfrak{R}, \dots, \mathfrak{R}}_{k-1}, {d}_\T \mathfrak{R}, \underbrace{\mathfrak{R}, \dots, \mathfrak{R}}_{N-k-1}\bigr),
\end{aligned}
\end{equation*}
which telescopes with the the right hand side of the first equation as we take the alternating sum over $N$, meaning that
\begin{equation*}
(\underline{d}_\T + \underline{b})\Ch(p) = 0.
\end{equation*}
Now $\underline{B}\Ch(p)$ is a sum of the terms $\underline{B}(p, \mathfrak{R}, \dots, \mathfrak{R})$, which are contained in ${\DD^{\T}(\Omega)}$, since $p$ is of degree zero.
\end{proof}

We proceed by discussing the homotopy invariance of the Bismut-Chern character. First some notation: If $(\theta_j^s)_{s \in [0, 1]}$, $j=0, \dots, N$ are continuously differentiable families of elements in $\Omega$, we set
\begin{equation*}
  \partial_s (\theta_{0}^s, \dots, \theta_{N}^s) := -\sum_{k=0}^{N} (-1)^{m_{k-1}}\bigl({\theta}_{0}^s, \dots, {\theta}_{k-1}^s, \sigma  \frac{\dd}{\dd s}{\theta}_{k}^s, \theta_{k+1}^s, \dots, \theta_{N}^s\bigr),
\end{equation*}
where $m_k = |\theta_0| + \dots + |\theta_k|-k$ and $\dot{\theta}_{k}^s$ denotes the derivative in direction of $s$.
Now given a continuously differentiable family $(p_s)_{s \in [0, 1]}$ of idempotents, we define the corresponding {\em Bismut-Chern-Simons form} by
\begin{equation*}
\mathrm{CS}\bigl((p_s)_{s \in [0, 1]}\bigr) := \int_0^1 \partial_s \Ch(p_s) \dd s.
\end{equation*}
As for $\Ch(p)$, it can be easily shown that $\Ch(p_s)$ is a continuously differentiable function of $s$ with values in $\CC^\epsilon(\Omega_\T)$ and that $\partial_s \Ch(p_s)$ is entire for any $s \in [0, 1]$. Now since this is a continuous function of $s$ with values in the complete locally convex space $\CC^\epsilon(\Omega_\T)$, its integral $\mathrm{CS}((p_s)_{s \in [0, 1]})$ is well-defined. 
%In the proof below, we use that the fundamental theorem of calculus holds for this $\CC^\epsilon(\Omega_\T)$-valued integral; for details, see \cite[Section~I.2]{KrieglMichor}

\begin{proposition}[Transgression] \label{PropTransgressionForm}
As elements of the Chen normalized complex $\NN^{\T, \epsilon}(\Omega)$, we have the transgression formula
\begin{equation*}
 \Ch(p_1) - \Ch(p_0) = (\underline{d}_\T+\underline{b}+\underline{B}){\mathrm{CS}} \bigl((p_s)_{s \in [0, 1]}\bigr).
\end{equation*}
\end{proposition}

\begin{proof}
We have to show that $(\underline{d}_\T+\underline{b}+\underline{B}){\mathrm{CS}} \bigl((p_s)_{s \in [0, 1]}\bigr)$ is contained in the closure $\overline{\DD^{\T}(\Omega)}$ of $\DD^\T(\Omega)$ in $\CC^\epsilon(\Omega_\T)$, for all $s\in [0,1]$. To this end, it is straightforward to verify that 
\begin{equation*}
\underline{d} \partial_s + \partial_s \underline{d} = \underline{b} \partial_s + \partial_{s}\underline{b} = \underline{B} \partial_s + \partial_{s}\underline{B}= 0, ~~~~\text{and}~~~~\underline{\iota} \partial_s + \partial_s \underline{\iota} = \frac{\dd}{\dd s},
\end{equation*}
where
\begin{equation*}
  \frac{\dd}{\dd s} (\theta_0^s, \dots, \theta_N^s) = \sum_{k=0}^N  (\theta_0^s, \dots, \frac{\dd}{\dd s}\theta_k^s, \dots, \theta_N^s)
\end{equation*}
is the derivative in the locally convex space $\CC^\epsilon(\Omega_\T)$. Hence we have
\begin{equation*}
\begin{aligned}
  (\underline{d}_\T+\underline{b}+\underline{B}){\mathrm{CS}}\bigl((p_s)_{s \in [0, 1]}\bigr) &= - \int_0^1 \partial_s (\underline{d}_\T+\underline{b}+\underline{B}) \Ch(p_s) \dd s + \int_0^1 \frac{\dd}{\dd s} \Ch(p_s) \dd s \\
  &=- \int_0^1 \partial_s \underline{B} \Ch(p_s) \dd s + \Ch(p_1) - \Ch(p_0),
  \end{aligned}
\end{equation*}
where in the last step, we used the fundamental theorem of calculus and the fact that $(\underline{d}_\T+\underline{b}) \Ch(p_s) = 0$ by the proof of Thm.~\ref{ThmBChClosed} above. 

It therefore remains to show that $\partial_s \underline{B} \Ch(p_s)$ is contained in $\overline{\DD^{\T}(\Omega)}$ for all $s$, which is not quite obvious as $\partial_s$ does not generally preserve $\overline{\DD^{\T}(\Omega)}$. It is easy to see that $\partial_s \underline{B} \Ch(p_s)$ is entire, so we just need an algebraic argument.
To this end, observe that modulo $\overline{\DD^{\T}(\Omega)}$, 
\begin{equation} \label{SimplerBofCh}
\begin{aligned}
 \underline{B}\Ch(p_s) &= \sum_{N=0}^\infty (-1)^N \sum_{k=0}^N \tr (\mathbf{1}, \underbrace{\mathfrak{R}, \dots, \mathfrak{R}}_k, \sigma \dot{p}, \underbrace{\mathfrak{R}, \dots, \mathfrak{R}}_{N-k}) \\
 &= \underline{B}\sum_{N=0}^\infty (-1)^N \tr (\sigma\dot{p}, \underbrace{\mathfrak{R}, \dots, \mathfrak{R}}_{N}),
\end{aligned}
\end{equation}
where we dropped the dependence on $s$ for readability and wrote $\dot{p} = \frac{\dd}{\dd s} p_s$.
Now because of $\dot{p} = \dot{p} p + p \dot{p}$ and the identity \eqref{IdentityInD}, we have modulo $\overline{\DD^{\T}(\Omega)}$
\begin{equation*}
\begin{aligned}
(\sigma\dot{p}, \underbrace{\mathfrak{R}, \dots, \mathfrak{R}}_{N})
= (p \sigma \dot{p}, \underbrace{\mathfrak{R}, \dots, \mathfrak{R}}_{N}) + (\sigma\dot{p}, d_\T p, \underbrace{\mathfrak{R}, \dots, \mathfrak{R}}_{N}) + (\sigma\dot{p}, p \mathfrak{R}, \underbrace{\mathfrak{R}, \dots, \mathfrak{R}}_{N-1}).
\end{aligned}
\end{equation*}
After summing over $N$ and shifting an index on the second factor, we can use \eqref{IdentityForP1} to obtain that modulo $\overline{\DD^{\T}(\Omega)}$,
\begin{equation*}
\sum_{N=0}^\infty (-1)^N (\sigma\dot{p}, \underbrace{\mathfrak{R}, \dots, \mathfrak{R}}_{N}) = 
\sum_{N=0}^\infty (-1)^N(p \sigma \dot{p}, \underbrace{\mathfrak{R}, \dots, \mathfrak{R}}_{N}) + \sum_{N=0}^\infty (-1)^N(\sigma\dot{p}, \mathfrak{R} p, \underbrace{\mathfrak{R}, \dots, \mathfrak{R}}_{N-1}).
\end{equation*}
Repeating this, we can shift the factor of $p$ in the second sum all the way to the right and then use \eqref{IdentityForP1} combined with \eqref{IdentityInD2}, making this term cancels with the first sum on the right hand side of the above equation. This implies
\begin{equation*}
\sum_{N=0}^\infty (-1)^N (\sigma\dot{p}, \underbrace{\mathfrak{R}, \dots, \mathfrak{R}}_{N}) \in \overline{\DD^{\T}(\Omega)}.
\end{equation*}
The claim now follows from \eqref{SimplerBofCh}, bearing in mind that $\underline{B}$ preserves $\overline{\DD^{\T}(\Omega)}$. 
\end{proof}

It is easy to see that $\Ch(p \oplus q) = \Ch(p) + \Ch(q)$, for projections $p \in \mathrm{Mat}_n(\Omega^0)$, $q \in \mathrm{Mat}_n(\Omega^0)$,  where $p \oplus q \in \mathrm{Mat}_{n+m}(\Omega^0)$. The transgression formula above therefore implies the following result.

\begin{theorem}
The Bismut-Chern character defines a group homomorphism
\begin{equation*}
\begin{tikzcd}
  K_0(\Omega^0) \ar[r] & H^+\bigl(\NN^{\T, \epsilon}(\Omega)\bigr)
  \end{tikzcd} 
\end{equation*}
from the $K$-theory of the algebra $\Omega^0$ to the cohomology of the complex $\NN^{\T, \epsilon}(\Omega)$.
\end{theorem}

\section{An Index Theorem}\label{SectionIndexTheorem}

In this section, let $\Omega$ be a locally convex dg algebra and let $\Mod=(\Hil, \cc, Q)$ be a  $\vartheta$-summable Fredholm module over $\Omega$. We emphasize that to obtain the results of this section, it is crucial that $\Mod$ satisfies \eqref{Multiplicativity}. We denote by $\Mod^{(n)} = (\Hil^n, \cc, Q)$ the $\vartheta$-summable Fredholm module over the dg algebra $\mathrm{Mat}_n(\Omega)$, where $Q$ acts diagonally on $\Hil^n$ and
\begin{equation*}
  \cc\begin{pmatrix} \theta_{11} & \dots & \theta_{1n} \\ \vdots & & \vdots \\ \theta_{n1} & \dots & \theta_{nn} \end{pmatrix} = \begin{pmatrix} \cc(\theta_{11}) & \dots & \cc(\theta_{1n}) \\ \vdots & & \vdots \\ \cc(\theta_{n1}) & \dots & \cc(\theta_{nn}) \end{pmatrix}.
\end{equation*}
 Here we identify $\Lin(\Hil^n) = \mathrm{Mat}_n(\Lin(\Hil))$. Similarly, we define the weak $\vartheta$-summable Fredholm module $\Mod_\T^n$ over $\mathrm{Mat}_n(\Omega_\T)$.

Given an idempotent $p=p^2 \in \mathrm{Mat}_n(\Omega^0)$, \eqref{Multiplicativity} implies that the operator $\cc(p)$ is an even idempotent acting on the Hilbert space $\Hil^n$, so that $\cc(p)\Hil^n$ becomes a closed subspace of $\Hil$ and thus a Hilbert space on its own right. As $\cc(p)$ is even, the $\Z_2$-grading restricts to $\cc(p)\Hil^n$. We obtain a well-defined odd unbounded operator 
\begin{equation*}
Q_p := \cc(p)Q\cc(p),
\end{equation*}
which is densely defined and closed on the Hilbert space $\cc(p)\Hil^n$.
The purpose of this section is to prove the following index theorem.

\begin{theorem}[Geometric index theorem] \label{ThmIndex} The operator $Q_p$ defined above is Fredholm and its $\Z_2$-graded index is given by the formula
\begin{equation} \label{IndexFormula}
  \mathrm{ind}(Q_p) = \Ch_{\Mod_\T} \bigl(\Ch(p)\bigr). 
\end{equation}
\end{theorem}

This theorem can be seen as a ``loop space version'' of the index theorem for the JLO-cocycle of Getzler-Szenes \cite{GetzlerSzenes}. While their proof relies on the homotopy invariance of the Chern character, a remarkable observation is that in our setup, the right hand side of \eqref{IndexFormula} is {\em directly} the super trace of a heat operator, without having to defer to any homotopy result. Precisely, we have the following proposition.

\begin{proposition}[Perturbation series] \label{Prop:PerturbationSeries}
Let $p \in \mathrm{Mat}_n(\Omega^0)$ be an idempotent. Then the unbounded operator
\begin{equation*}
  Q_1: = \cc(p)Q\cc(p) + \cc(p^\perp)Q\cc(p^\perp),
\end{equation*}
on $\Hil^n$ is closed on the same domain as $Q$, and its square generates a strongly continuous semigroup of operators. Moreover, for all $T>0$, one has the formula
\begin{equation} \label{PerturbationSeries}
  e^{-TQ_1^2} = \sum_{N=0}^\infty (-1)^N\Phi_{T}^{\M_\T^n} (\underbrace{\mathfrak{R}_p, \dots, \mathfrak{R}_p}_{N}).
\end{equation}
\end{proposition}

\begin{proof}
  From \eqref{Multiplicativity}, we obtain $[Q, \cc(p)] = \cc(dp)$ and $\cc(p)\cc(dp) = \cc(p\,dp)$. Therefore
  \begin{equation*}
  Q_1 = Q + \cc(2p-\mathbf{1})[Q, \cc(p)] = Q + \cc\bigl((2p-\mathbf{1})dp\bigr).
  \end{equation*}
This shows that that $Q_1$ is in fact a bounded perturbation of $Q$, hence is a closed operator on the same domain as $Q$.
On the other hand, we have $p\,dp + p^\perp d p^\perp = (2p-\mathbf{1})dp$, so with a view on \eqref{ComponentsOfExtendedF}, 
  \begin{equation*}
  \begin{aligned}
  F(\mathfrak{R}) &= \cc\bigl(d\bigl\{(2p-\mathbf{1})dp\bigr\}\bigr) - \bigl[Q, \cc\bigl((2p-\mathbf{1})dp\bigr)\bigr] -  \cc\bigl((dp)^2\bigr)\\
  &= \cc\bigl((dp)^2\bigr) - \bigl[Q, \cc\bigl((2p-\mathbf{1})dp\bigr)\bigr],
  \end{aligned}
  \end{equation*}
  as well as 
  \begin{equation*}
  \begin{aligned}
    F(\mathfrak{R}, \mathfrak{R}) &= \cc\bigl((2p-\mathbf{1})dp\bigr)^2 - \cc\bigl(((2p-\mathbf{1})dp)^2\bigr)\\
    &=  \cc\bigl((2p-\mathbf{1})dp\bigr)^2 +\cc\bigl((dp)^2\bigr).
  \end{aligned}
  \end{equation*}
Putting together, we obtain
\begin{equation*}
Q^2 - F(\mathfrak{R}) + F(\mathfrak{R}, \mathfrak{R}) = Q^2 + \bigl[Q, \cc\bigl((2p-\mathbf{1})dp\bigr)\bigr] + \cc\bigl((2p-\mathbf{1})dp\bigr) = Q_1^2.
\end{equation*}
Hence we can write $Q_1^2 = Q^2 + R$, where $R: = F(\mathfrak{R}, \mathfrak{R}) - F(\mathfrak{R})$. By the assumptions (A1), (A2) on the Fredholm module, we can use Lemma~\ref{LemmaMainEstimate} to define $e^{-TQ_1^2} = e^{-T(Q^2 + R)}$ using the perturbation series
\begin{equation} \label{PerturbationSeriesIndexSection}
  e^{-T(Q^2 + R)} = \sum_{M=0}^\infty (-T)^M \bigl\{\underbrace{R, \dots, R}_M\bigr\}_{TQ^2}.
\end{equation}
It is easy to see that this indeed defines a strongly continuous semigroup of operators with infinitesimal generator $Q^2+R = Q_1^2$. Now we have
\begin{equation*}
\sum_{M=0}^\infty (-T)^M \{\underbrace{R, \dots, R}_{M}\}_{TQ^2} = \sum_{M=0}^\infty (-T)^M\bigl\{F(\mathfrak{R}, \mathfrak{R}) - F(\mathfrak{R}), \dots, F(\mathfrak{R}, \mathfrak{R}) - F(\mathfrak{R})\bigr\}_{TQ^2}.
\end{equation*}
Expanding each term on the right hand side of this equation by multi-linearity, we see that we obtain an infinite sum of brackets $\{\dots\}_{TQ^2}$ which contains all possible sequences of operators $F(\mathfrak{R})$ and $F(\mathfrak{R}, \mathfrak{R})$. Observe now that by formula \eqref{ExplicitPhiT}, this coincides with the right hand side of \eqref{PerturbationSeries}.
\end{proof}

\begin{proof}[of Thm.~\ref{ThmIndex}]
If $\cc(p)$ is self-adjoint, the proof follows from Prop.~\ref{Prop:PerturbationSeries}. 
Namely, because
\begin{equation*}
  \Str_{\Hil} \Phi^{\Mod_\T}\bigl(\tr(\theta_1, \dots, \theta_N) \bigr) =  \Str_{\Hil^n} \Phi^{\Mod_\T^n}(\theta_1, \dots, \theta_N),
\end{equation*}
we get by using the definition \eqref{FormulaBChP} of $\Ch(p)$ and the formula \eqref{DefinitionChernCharacterMod} for $\Ch_{\Mod_\T}$ that
\begin{equation*}
\begin{aligned}
 \Ch_{\Mod_\T}\bigl(\Ch(p)\bigr) &= \sum_{N=0}^\infty (-1)^N \Str \bigl( \cc(p) \Phi_{T}^{\M_\T^n} (\underbrace{\mathfrak{R}_p, \dots, \mathfrak{R}_p}_{N})\bigr)\\
&= \Str_{\Hil^n}\bigl(\cc(p) e^{-Q_1^2}\bigr) \\
&= \Str_{\cc(p)\Hil^n} (e^{-Q_p^2}),
\end{aligned}
\end{equation*}
which equals the graded index of $Q_p$ by the usual McKean-Singer formula. 

%This equals the graded index of $Q_p$ by the usual argument of McKean-Singer: Consider 
%\begin{equation*}
%  f(T) = \Str_{\cc(p)\Hil^n} (e^{-TQ_p^2}).
%\end{equation*}
%Then since $Q_p$ commutes with $e^{-TQ_p^2}$ and $Q_p$ is odd,  we obtain
%\begin{equation*}
%f^\prime (T) = \Str_{\cc(p)\Hil^n} (e^{-TQ_p^2} Q_p^2) = \frac{1}{2}\Str_{\cc(p)\Hil^n} ([e^{-TQ_p^2} Q_p, Q_p]) = 0, 
%\end{equation*}
%where we used the supertrace property. Hence $f(T)$ is in fact constant in $T$. We therefore obtain
%\begin{equation*}
%\bigl\langle \Ch(\Mod_\T), \Ch(p)\bigr\rangle = \lim_{T \to \infty} f(T) = \Str_{\cc(p)\Hil^n} (P_0),
%\end{equation*}
%where $P_0$ is some idempotent with range $\ker(Q_p) = \ker(Q_p)$. 

To deal with the general case, we adapt an idea of Getzler-Szenes \cite{GetzlerSzenes} to our setup, making the following construction. Consider the $*$-subalgebra
\begin{equation} \label{DefinitionAlgebraB}
\begin{aligned}
  \mathcal{B}_Q := \bigl\{ A \in \Lin(\Hil)^+ \mid \Delta^{1/2} A \Delta^{-1/2} &~~\text{and}~~ \Delta^{1/2} A^* \Delta^{-1/2}~~\\
  &\text{are densely defined and bounded}\bigr\},
\end{aligned}
\end{equation}
of $\Lin(\Hil)$, where $\Delta = Q^2+1$ and $\Lin(\Hil)^+$ denotes the subalgebra of even operators. 
%Then $\mathcal{B}_Q$ is a Banach $*$-algebra with the norm
%\begin{equation*}
%  \|A\|_Q := \|\Delta^{1/2} A \Delta^{-1/2}\| + \|A\|.
%\end{equation*}
Note that we have
\begin{equation*}
  (Q+A)^2 = Q^2 + [Q, A] + A^2,
\end{equation*}
hence for $A \in \mathcal{B}_Q$, the operator $R = [Q, A] + A^2$ satisfies the assumptions of Lemma~\ref{LemmaMainEstimate}. We obtain that for any $A \in \mathcal{B}_Q$, the heat semigroup $e^{-T(Q+A)^2}$ can be defined using the perturbation series \eqref{PerturbationSeriesIndexSection}.

Furthermore, set
\begin{equation*}
  \A_Q := \bigl\{ A \in \Lin(\Hil) \mid [Q, A] \in \mathcal{B}_Q \bigr\}
\end{equation*}
and let $\Omega_Q = \Omega_{\A_Q}$ the corresponding dg algebra of non-commutative differential forms, which comes with a canonical map $\cc_Q: \Omega_Q \rightarrow \Lin(\Hil)$ so that $\Mod_Q = (\Hil, \cc_Q, Q)$ is a Fredholm module over $\Omega_Q$, as explained in Example~\ref{ExampleNoncommutativeForms}. 

Now the calculations in the proof of Prop.~\ref{Prop:PerturbationSeries} show that for an idempotent $p \in \Omega^0$, we have
\begin{equation*}
   \Ch_{\Mod_\T}\bigl( \Ch(p)\bigr) =  \Ch_{\Mod_{Q, \T}}\bigl(\Ch\bigl(\cc(p)\bigr)\bigr),
\end{equation*}
where the right hand side denotes the pairing of the extended entire cyclic complex associated to $\Omega_Q$. Therefore, it suffices to prove Thm.~\ref{ThmIndex} for the ``universal example'' $\Omega_Q$. 

To this end, we will show that the idempotent $\cc(p)$ can be homotoped within $\A_Q$ to a self-adjoint projection $P$, as then
\begin{equation*}
 \Ch_{\Mod_{Q, \T}}\bigl( \Ch\bigl(\cc(p)\bigr)\bigr) =  \Ch_{\Mod_{Q, \T}}\bigl( \Ch(P)\bigr)
\end{equation*}
by the transgression formula Thm.~\ref{PropTransgressionForm} and the closedness of $\Ch_{\Mod_{Q, \T}}$, Thm.~\ref{ThmChClosed} (together with the fact that it is Chen normalized, Thm.~\ref{ThmNormalized}). By Prop.~4.6.2 in \cite{BlackadarKTheory}, this property follows from the fact that $\A_Q$ is a {\em spectral} subalgebra of $\Lin(\Hil)$, i.e., that for any $A \in \A_Q$ which is invertible in $\Lin(\Hil)$, its inverse $A^{-1} \in \A_Q$. To see that the larger algebra $\mathcal{B}_Q$ is spectral, notice that for $A \in \mathcal{B}_Q$,
\begin{equation*}
 \Delta^{1/2} A^{-1} \Delta^{-1/2} = (\Delta^{1/2} A \Delta^{-1/2})^{-1}.
\end{equation*}
But $\Delta^{1/2} A \Delta^{-1/2}$ is bounded, hence its inverse is bounded as well. Now for $A \in \A_Q$, we have
\begin{equation*}
  \bigl[Q, A^{-1}\bigr] = A^{-1}[Q, A]A^{-1},
\end{equation*}
which is a product of elements in $\mathcal{B}_Q$ by assumption, hence in $\mathcal{B}_Q$. Since these arguments also work for $A^*$ instead of $A$, this shows that $\A_Q$ is indeed spectral in $\Lin(\Hil)$ and finishes the proof.
\end{proof}

Thm.~\ref{ThmIndex} in combination with Thm.~\ref{PropTransgressionForm} and homotopy invariance of the index easily implies the following result.

\begin{corollary}[Homological index theorem]\label{ThmHomologicalIndex} We have a commutative diagram
\begin{equation} \label{KHom}
\begin{tikzcd}[row sep=large]
  K_0(\Omega^0) \ar[d, "\mathrm{ind}_\Mod"'] \ar[r, "\Ch"] & H^+\bigl(\NN^{\T, \epsilon}(\Omega)\bigr) \ar[d, "\Ch_{\Mod_\T}"] \\
  \Z \ar[r] & \C,
  \end{tikzcd} 
\end{equation}
where $\mathrm{ind}_\Mod$ is the map that assigns the index of the operator $Q_p$ on $\cc(p) \Hil$.
\end{corollary}

Notice in particular that if $\A$ is a locally convex algebra, we have such a commutative diagram \eqref{KHom} for any dg algebra $\Omega$ with $\Omega^0 = \A$ together with a $\vartheta$-summable Fredholm module $\Mod$ over $\Omega$.

\section{The Supersymmetric Path Integral and its Localization Property} \label{SectionPathIntegral}

In this section, we finally discuss the application of our abstract theory to the construction of the supersymmetric path integral for the $\mathcal{N}=1/2$ supersymmetric $\sigma$-model and its localization property, which originally was our main motivation. 

Let $X$ a compact spin manifold of even dimension $n$. In the considerations of Atiyah \cite{AtiyahCircular}, the corresponding supersymmetric path integral is an integration functional $I$ for suitable differential forms on the smooth loop space $\L X$ of $X$, formally defined by \eqref{IntroFormal}. The aim of this section is to use the abstract results from the previous sections to rigorously construct this linear functional and to establish its two fundamental properties: supersymmetry and its localization property \eqref{LocalizationFormulaIntro}.

\medskip

To this end, consider the dg algebra $\Omega = \Omega(X)$ of differential forms over $X$. There is a canonical Fredholm module $\Mod^X  = (\Hil, Q, \cc)$ over $\Omega(X)$, as discussed in Example~\ref{ExampleSpinors}: $\Hil = L^2(X, \Sigma)$, the space of square-integrable sections of the spinor bundle $\Sigma$, $Q = \DD$, the Dirac operator, and $\cc$ is given by Clifford multiplication \eqref{QuantizationMapEx}. This Fredholm module satisfies the identities \eqref{Multiplicativity}, hence according to Thm.~\ref{ThmNormalized}, the Chern character $\Ch_{\Mod^X_\T}$ is Chen normalized, i.e., descends to a cochain on the normalized complex $\NN^{ \T, \epsilon}(\Omega(X))$ and itself defines an element in $\NN_{\T, \alpha}^+(\Omega(X))$. The following result is key in this situation; it seems likely that it can be generalized to more general situations analogously to the local index formula of Connes and Moscovici, c.f.\ \cite{ConnesMoscoviciLocal, HigsonLocal}.

\begin{theorem}[Localization] \label{ThmLocalizationFormula}
As an element of the complex $\NN_{\T, \alpha}^+(\Omega(X))$, the Chern character $\Ch_{\Mod^X_\T}$ is cohomologous to the cochain $\mu_0$ defined by
\begin{equation} \label{ChainC_0}
  \mu_0(\theta_0, \dots, \theta_N) = \frac{1}{(2 \pi i)^{n/2}N!} \int_X \hat{A}(X) \wedge \theta_0^\prime \wedge \theta_1^{\prime\prime} \wedge \cdots \wedge \theta_N^{\prime\prime}
\end{equation}
for $\theta_0, \dots, \theta_N \in \Omega(X)_\T$. 
\end{theorem}

Here 
\begin{equation} \label{AHatForm}
\hat{A}(X) = \det\nolimits^{1/2}\left(\frac{R/2}{\sinh(R/2)}\right)
\end{equation}
 is the Chern-Weil representative of the $\hat{A}$-form of $X$, involving the Riemann curvature tensor $R$ of $X$; c.f.\ {\normalfont \cite[(1.35)]{BGV}}.
Direct computations show that $\underline{d}\mu_0 = 0$ (this uses the fact that $\hat{A}(X)$ is closed together with Stokes theorem), $\underline{b} \mu_0 = 0$ and $\underline{B} \mu_0 = \underline{\iota}\mu_0$, so that $(\underline{d}_\T + \underline{b} + \underline{B})\mu_0 = 0$. Moreover, it is easy to check that $\mu_0$ is both analytic and Chen normalized, hence indeed is a closed element of the complex $\NN_{\T, \alpha}^+(\Omega(X))$.

\begin{proof}
We define a homotopy of $\vartheta$-summable Fredholm modules $\Mod^{X, T} = (\Hil, Q_T, \cc_T)$, $T > 0$,  over $\Omega(X)$ by setting
\begin{equation*}
  Q_T = T^{1/2}\DD, \qquad \cc_T(\theta) = T^{|\theta|/2}\cc(\theta).
\end{equation*}
One checks that the scaling is such that the relations \eqref{Multiplicativity} are preserved, hence we obtain Chen normalized Chern characters $\Ch_{\Mod^{X, T}_\T} \in \NN_{\T, \alpha}(\Omega(X))$, which are all cohomologous by Thm.~\ref{TheoremHomotopyInvariance}. The theorem now follows from the following proposition, which generalizes the result of Block and Fox \cite[Thm.~4.1]{BlockFox} on the JLO-cocycle to our situation; c.f.\ also the calculations in \cite{ConnesMoscoviciNovikov}.
\end{proof}

\begin{proposition}
For all $\theta_0, \dots, \theta_N \in \Omega(X)_\T$, we have
\begin{equation*}
  \lim_{T \rightarrow 0} \Ch_{\Mod^{X, T}_\T} (\theta_0, \dots, \theta_N) = \mu_0(\theta_0, \dots, \theta_N).
\end{equation*}
\end{proposition}

Below we sketch how this can be proved using Getzler rescaling. A full proof using a rescaled spinor bundle of Connes' tangent groupoid is given in \cite{ludewig2020short}.

\begin{proof}[Sketch]
Explicit calculations show that if $\theta_0, \dots, \theta_N \in \Omega(X)_\T$ are homogeneous (with, say $|\theta_k| = \ell_k$), we have
\begin{equation} \label{Substitution}
\begin{aligned}
  &\Ch_{\Mod^{X, T}_\T} (\theta_0, \dots, \theta_N) = T^{m_N/2} \Str\bigl(\cc(\theta_0^\prime) \Phi_T^{\M^X_\T}(\theta_1, \dots, \theta_N)\bigr) \\
  &= \sum_{M=1}^N (-T)^M \sum_{\mathscr{P}_{M, N}} \int_{\Delta_M} \Str \Bigl( \cc(\theta_0^\prime)e^{-T\tau_1 \DD^2} F(\theta_{I_1}) e^{-T(\tau_2-\tau_1)\DD^2} \cdots F(\theta_{I_M}) e^{-T(1-\tau_M)\DD^2} \Bigr) \dd \tau
  \end{aligned}
\end{equation}
where $m_N = \ell_0 + \dots + \ell_N - N$.
The right hand side is now amenable to Getzler rescaling \cite{GetzlerRescaling, BGV}; we assume for the rest of the proof that the reader is familiar with this technique. To calculate the short-time limit of the right hand side of \eqref{Substitution}, we need to calculate the orders of the operators $F(\theta_k)$, $F(\theta_k, \theta_{k+1})$ in the Getzler filtration. As usual, write $\theta = \theta^\prime + \sigma \theta^{\prime\prime}$ so that $|\theta_k^\prime| = \ell_k$, $|\theta_k^{\prime\prime}|=\ell_k+1$. Since Clifford multiplication with with a differential form $\theta$ has order $|\theta|$ in the Clifford filtration and the Dirac operator has degree two, the operator 
\begin{equation} \label{FormulaFirstF}
  F(\theta_k) = \cc(d\theta_k^\prime) - [\DD, \cc(\theta_k^\prime)] - \cc(\theta_k^{\prime\prime})
\end{equation}
has degree $\ell_k+1$. However, an explicit calculation of the commutator (see Lemma~4.12 in \cite{HanischLudewig1}) shows that the terms involving $\theta_k^\prime$ are actually  of order $\ell_k$ only; the formula is
\begin{equation*}
  \cc(d\theta^\prime_k) - [\DD, \cc(\theta^\prime_k)] = 2\sum_{j=1}^n \cc(\iota_{e_j}\theta_k^{\prime}) \nabla_{e_j} - \cc(d^*\theta_k^\prime),
\end{equation*}
in terms of a local basis $e_1, \dots, e_n$ of $TX$, where $d^*$ is the codifferential. This means that only term involving $\theta_k^{\prime\prime}$ of \eqref{FormulaFirstF} is relevant for calculation of the short-time limit. The term 
\begin{equation*}
  F(\theta_k, \theta_{k+1}) = (-1)^{|\theta_k|} \bigl( \cc(\theta_k)\cc(\theta_{k+1}) - \cc(\theta_k \wedge \theta_{k+1})\bigr)
\end{equation*}
 has degree $\ell_k + \ell_{k+1}$ in the Getzler filtration. This shows that the term of highest order in the Getzler on the right hand side of \eqref{Substitution} is the summand with $M=N$; it has order $\ell_0 + \dots + \ell_N + N$. Using the explicit formula \eqref{ExplicitPhiT} for $\Phi_T$ and the previous observation that only the terms involving $\theta_k^{\prime\prime}$ contribute for $k \geq 1$, we therefore obtain that modulo higher order in $T$, we have
\begin{equation*}
\begin{aligned}
\Str\bigl(\cc(\theta_0^\prime) \Phi_T^{\M^X_\T}(\theta_1, \dots, \theta_N)\bigr) ~\sim~ (-T)^N \Str\Bigl(\cc(\theta_0^\prime) \bigl\{- \cc(\theta_1^{\prime\prime}), \dots, - \cc(\theta_N^{\prime\prime})\bigr\}_{Q_T^2}\Bigr).
 \end{aligned}
\end{equation*}
Getzler's technique now allows to compute this term explicitly, up to higher order in $T$; the result is
\begin{equation*}
\Str\Bigl(\cc(\theta_0^\prime) \bigl\{\cc(\theta_1^{\prime\prime}), \dots, \cc(\theta_N^{\prime\prime})\bigr\}_{Q_T^2}\Bigr)~ \sim ~ \frac{T^{-m_N/2-N}}{(2\pi i)^{n/2}N!} \int_X \hat{A}(X) \wedge \theta_0^\prime \wedge \theta_1^{\prime\prime} \wedge \cdots \wedge \theta_N^{\prime\prime},
\end{equation*} 
where the additional $N!$ in the denominator comes from evaluating the integral over $\Delta_N$ contained in the bracket $\{ \cdots \}_{Q_T^2}$ (c.f.\ Thm.~4.2 in \cite{BGV} and \cite{BlockFox}). Substituting this into \eqref{Substitution} shows that indeed, the small-$T$-limit is given by $\mu_0(\theta_0, \dots, \theta_N)$.
\end{proof}

With this result at hand, we now use our theory to  construct the desired path integral map $I$, which will be an integration functional for differential forms on the loop space $\L X$. In other words, it will be a linear functional, defined on a suitable space of {\em integrable differential forms} on $\L X$.

Our theory connects to the loop space through Chen's iterated integral map \cite{Chen1, Chen2}, or rather its extended version, constructed by Getzler, Jones and Petrack \cite{GJP}, which is a map
\begin{equation} \label{IteratedIntegralMap}
\rho: \CC\bigl(\Omega_\T(X)\bigr) \longrightarrow \Omega(\L X)=\bigoplus^{\infty}_{k=0}\Omega^k(\L X).
\end{equation}
In order to provide a definition that is both concise and explicit, given $\theta \in \Omega^k(X)$ and $\tau \in S^1$ we denote by $\theta(\tau) \in \Omega^k(\L X)$ the differential form defined at $\gamma \in \L X$ by the formula
\begin{equation*}
  \theta(\tau)_\gamma [v_1, \dots, v_k] = \theta_{\gamma(\tau)}\bigl[v_1(\tau), \dots, v_k(\tau)\bigr],
\end{equation*}
for $v_1, \dots, v_k \in T_\gamma \L X = C^\infty(S^1, \gamma^*TX)$. The (extended) iterated integral map is then given by the formula
\begin{equation*}
  \rho(\theta_0, \dots, \theta_N) = \int_{\Delta_N} \theta_0^\prime (0) \wedge \bigl(\iota_K\theta^\prime_1(\tau_1) - \theta^{\prime\prime}(\tau_1)\bigr) \wedge \cdots \wedge \bigl(\iota_K\theta^\prime_N(\tau_N) - \theta^{\prime\prime}(\tau_N)\bigr)  \dd \tau,
\end{equation*}
where $K(\gamma) = \dot{\gamma}$ denotes the canonical velocity vector field on $\L X$. 

\begin{lemma}[Properties of $\rho$] \label{LemmaPropRho}
The extended iterated integral map is degree-preserving and satisfies
\begin{equation} \label{IdentitiesRhoChainMap}
\rho(\underline{d}+\underline{b}) = d \rho, \qquad \rho\underline{B} - \mathbf{A}\rho\underline{\iota} = \mathbf{A}\iota_K \rho,
\end{equation}
where $\mathbf{A}$ is the operator on $\Omega(\L X)$ that averages over the $S^1$-action. Moreover, if $X$ is connected, we have $\ker(\rho)\subseteq\ker(\Ch_{\Mod^X})$. 
\end{lemma}

\begin{proof}
The identities \eqref{IdentitiesRhoChainMap} follow from tedious, but rather straight-forward calculations; c.f.\ the proof of Prop.~4.1 in \cite{CacciatoriGueneysu}. Now if $X$ is connected, it turns out that the kernel of $\rho$ is precisely the subcomplex $\DD^\T(\Omega)$ defined in \eqref{DefinitionDT}, so the Lemma follows from the fact that $\Ch_{\Mod^X_\T}$ is Chen normalized. The inclusion $\DD^\T(\Omega) \subseteq \ker(\rho)$ is a simple calculation; the converse inclusion is a generalization of Lemma~4.1 of \cite{Chen1}, which follows from observations similar to those made in \cite[\S2]{GJP}.
\end{proof}

Assume from now on that $X$ is connected; then these results on $\rho$ allow to define the desired integration functional for differential forms on the loop space. It is defined on the subclass of differential forms which can be represented by iterated integrals, in other words, on the subspace $\mathrm{im}(\rho)\subset \Omega(\L X)$. The definition is
\begin{equation} \label{DefinitionOfI}
  I[\xi] := i^{n/2} \Ch_{\M^X_\T}(c), \qquad \text{if} ~ \xi = \rho(c) ~ \text{for} ~ c \in \CC\bigl(\Omega(X)_\T\bigr),
\end{equation}
where the factor $i^{n/2}$ ensures that $I$ is real-valued on real forms $\xi$. Notice that if $\xi$ is represented as an iterated integral in two ways, say by $c, c^\prime \in \CC(\Omega(X)_\T)$, then one has $c-c^\prime \in \ker(\rho)$ and Lemma~\ref{LemmaPropRho} implies that $\Ch_{\M_\T^X}(c-c^\prime) = 0$, as $\Ch_{\M_\T^X}$ is Chen normalized. In other words, $I[\xi]$ is independent of the choice of $c$.

\begin{proposition}
For any $\xi \in \Omega(\L X)$ in the image of the equivariant iterated integral map, we have
\begin{equation} \label{StokesProperty}
  I\bigl[(d-\iota_K) \xi\bigr] = 0.
\end{equation}
\end{proposition}

\begin{proof}
The graded cyclicity of the supertrace implies that $I$ is cyclically invariant, i.e., $I[\xi] = I[\mathbf{A}\xi]$ for all iterated integrals $\xi$ (not that $\mathbf{A}$ preserves $\im(\rho)$). From \eqref{IdentitiesRhoChainMap}, it therefore follows that the {\em equivariant iterated integral map} $\mathbf{A}\rho$ intertwines $\underline{d}_\T + \underline{b} + \underline{B}$ with $d-\iota_K$. Hence \eqref{StokesProperty} follows from the closedness of the Chern character, Thm.~\ref{ThmChClosed}.
\end{proof}

\begin{remark}
In several places, for example \cite{CacciatoriGueneysu, GJP}, the equivariant differential $d+P$ is used instead, where $P = \mathbf{A} \iota_K$. This difference is immaterial in our considerations, as $I[\mathbf{A}\xi] = I[\xi]$. Therefore, the Stokes property \eqref{StokesProperty} holds also with $P$ instead of $\iota_K$, and Thm.~\ref{ThmLocalization2} also holds for all iterated integrals $\xi$ with $(d+P)\xi =0$.
\end{remark}

Particularly interesting integrands for the path integral map are the {\em Bismut-Chern characters} $\mathrm{BCh}(E, \nabla)$ defined in \cite{Bismut1}, which are equivariantly closed differential forms on $\L X$ associated to a vector bundle $E$ with connection $\nabla$ over $X$. However, they are not contained in the algebra $\Omega(\L X)$, as they are an \emph{infinite} sum of their homogenous summands. To extend the iterated integral map to this context, we use Lemma~3.1 of \cite{HanischLudewig1} (see also Prop.~4.1 in \cite{CacciatoriGueneysu}), which implies that the iterated integral map \eqref{IteratedIntegralMap} extends by continuity to a map
\begin{equation*}
\rho_\epsilon: \CC^\epsilon\bigl(\Omega(X)_\T\bigr) \longrightarrow \widehat{\Omega}(\L X) :=  \prod_{k=0}^\infty \Omega^k(\L X),
\end{equation*}
where $\Omega^k(\L  X)$ is equipped with its canonical locally convex topology and $\widehat{\Omega}(\L X)$ with the induced product topology. Here, the locally convex topology on $\Omega^k(\L  X)$ is given by the family of seminorms $\nu_f(\xi) :=\nu(f^*\xi)$, where $f$ is a smooth a map from a finite dimensional manifold $Y$ to $\L X$ and $\nu$ is a continuous seminorm on $\Omega^k(Y)$ with respect to its natural Fr\'echet topology. This is much in the spirit of Chen's theory of diffeologies \cite{Chen1}.

Now, using $\rho_{\epsilon}$ instead of $\rho$ allows to extend the definition \eqref{DefinitionOfI} of $I$ to the differential forms $\xi$ that are {\em entire} iterated integrals, i.e. that lie in the image of $\rho_\epsilon$, so that with 
\begin{equation*}
\Omega_{\mathrm{int}}(\L X) := \mathrm{im}(\rho_\epsilon)\subset \widehat{\Omega}(\L X),
\end{equation*}
the space of \emph{integrable differential forms on $\L X$}, the linear functional $I$ given by (\ref{DefinitionOfI}) extends to a linear functional on $\Omega_{\mathrm{int}}(\L X)$.

\begin{theorem}[Fundamental properties of $I$] \label{ThmLocalization2} Let $X$ be a compact even-dimensional connected spin manifold. Then the linear functional 
\begin{equation*} \label{supint}
I: \Omega_{\mathrm{int}}(\L X) \longrightarrow \C
\end{equation*}
defined above is even and supersymmetric. Moreover, for all differential forms $\xi \in \Omega_{\mathrm{int}}(\L X)$ with  $(d-\iota_K)\xi=0$, one has the localization formula
\begin{equation} \label{LocFormulaEq}
  I[\xi] = (2\pi)^{-n/2} \int_X \hat{A}(X) \wedge \xi|_{X}.
\end{equation}
\end{theorem}

We remark that by the observations of Atiyah \cite{AtiyahCircular}, this is precisely the formula one gets when formally applying the finite-dimensional localization formula, c.f.\ e.g., \cite{DuistermaatHeckmann, BerlineVergne} or \cite[Thm.~7.13]{BGV}.

\begin{proof}
That $I$ is even follows from the corresponding property of the Chern character; the supersymmetry was already noted in \eqref{StokesProperty} and, via Stokes theorem on $X$, is in fact also a consequence of \eqref{LocFormulaEq}. It therefore remains to prove the localization formula, which follows essentially from Thm.~\ref{ThmLocalizationFormula}. Indeed, let $\xi = \rho_\epsilon(c)$ for $c \in \CC^\epsilon(\Omega(X)_\T)$, so that
\begin{equation*}
I[\xi] = \Ch_{\Mod_\T^X}(c) =  (\Ch_{\Mod_\T^X} - \mu_0)(c) +  \mu_0(c).
\end{equation*}
Now by Thm.~\ref{ThmLocalizationFormula}, $\Ch_{\Mod_\T^X}$ and $\mu_0$ are cohomologous, in other words,  there exists a Chen normalized cochain $\mu^\prime \in \NN_{\T, \alpha}^-(\Omega)$ such that $\Ch_{\Mod_\T^X}- \mu_0 = (\underline{d}+\underline{b} + \underline{B})\mu^\prime$. Therefore
\begin{equation} \label{Eq203948}
(\Ch_{\Mod_\T^X} - \mu_0)(c) = - \mu^\prime\bigl((\underline{d}_\T+\underline{b} + \underline{B})c\bigr).
\end{equation}
However, since $\xi$ is closed, Lemma~\ref{LemmaPropRho} implies that $(\underline{d}+\underline{b} + \underline{B})c \in \DD^\T(\Omega)$, which shows that the right hand side of \eqref{Eq203948} is zero, as $\mu^\prime$ is Chen normalized. We therefore have $I[\xi] = \mu_0(c)$, and the result follows after observing that 
\begin{equation*}
 \rho(\theta_0, \dots, \theta_N)\bigr|_X = \frac{1}{N!} \theta_0^\prime \wedge \theta_1^{\prime\prime} \wedge \cdots \wedge \theta_N^{\prime\prime},
\end{equation*}
where the factor $1/N!$ comes from the integral over $\Delta_N$ in the definition of $\rho$.
\end{proof}

\begin{remark}
To see how this connects to the twisted Atiyah-Singer index theorem, let $(E, \nabla)$ be a vector bundle with connection on $X$ and let $\mathrm{BCh}(E, \nabla)$ be the corresponding Bismut-Chern character, as defined in \cite{Bismut1}. Then by the results of \cite{GJP}, one has $\mathrm{BCh}(E, \nabla) = \rho_\epsilon(\Ch(p))$ for a suitable idempotent $p$, hence a corollary of Thm.~\ref{ThmIndex} is the formula
\begin{equation} \label{IndexTheoremIntroLoop}
  I\big[\mathrm{BCh}(E, \nabla)\bigr] = i^{n/2} \Ch_{\Mod^X_\T}\bigl( \Ch(p) \bigr) = i^{n/2}\mathrm{ind}(\DD^E), 
\end{equation}
where $\DD^E$ is the Dirac operator twisted by $(E, \nabla)$.  On the other hand, the Bismut-Chern-characters are equivariantly closed, hence by Thm.~\ref{ThmLocalization2}, we have
\begin{equation} \label{LocalizationBCh}
  I\big[\mathrm{BCh}(E, \nabla)\bigr] = (2 \pi)^{-n/2}\int_X \hat{A}(X) \wedge \mathrm{BCh}(E, \nabla)|_X .
\end{equation}
The Bismut-Chern characters now have the property that $\mathrm{BCh}(E, \nabla)|_X = \ch(E, \nabla)$, the usual Chern character on $X$ of $(E, \nabla)$, defined using Chern-Weil theory. Hence combining \eqref{LocalizationBCh} and \eqref{IndexTheoremIntroLoop}, we get the Atiyah-Singer index formula
\begin{equation*}
  \ind(\DD^E) = (2 \pi i)^{-n/2} \int_X \hat{A}(X) \wedge \ch(E, \nabla)
\end{equation*}
for the twisted Dirac operator. 
\end{remark}

%This puts the formal arguments of \cite{AtiyahCircular, Bismut1} on a solid mathematical ground and (together with the results of \cite{HanischLudewig1, HanischLudewig2}) justifies to call $I$ the \emph{supersymmetric path integral} associated to $X$.

\bibliography{literature}

\end{document}